\documentclass[11pt]{amsart}
\usepackage{mathrsfs}
\usepackage{amsfonts,amssymb,amscd,amsthm,amsmath,color}
\usepackage[alphabetic]{amsrefs}
\usepackage[all]{xy}
\makeatletter
\def\makeop#1{\expandafter\def\csname#1\endcsname{\mathop{\rm #1}\nolimits}\ignorespaces}
\def\makeoplist#1 {\def\@@tmpa{#1}\def\@@tmpb{***}%
  \ifx\@@tmpa\@@tmpb\else\makeop{#1}\expandafter\makeoplist\fi}
\makeoplist
ad sc gl Ad Alb Alt Aut Br can Card card Char Cl Coker coker Cond cond Corr
depth diag Diff Disc disc Div div dom End Ext Fil Fr Frac Frob Gal GL
Gr gr GSp GSpin Hom hom Id id Im im Ind ind Int inv Irr irr Isom
Ker ker len length Lie Nr Nrd O Ob ob ord PGL Pic Pr pr Proj PSL rank
RD Res Sgn sgn SL SO Sp Sp sp Spec Spf Spin
St st Stab stab Std std SU Succ Supp supp Swan Sym Tor tor tors
torsion Tr tr Trd U val Vol vol ***
\def\makermlist#1 {\def\@@tmpa{#1}\def\@@tmpb{***}\ifx\@@tmpa\@@tmpb
  \else\expandafter\def\csname#1\endcsname{{\rm#1}}\expandafter\makermlist\fi}
\makermlist ab alg arith can geom opp op red sep sh ***
\makeatother
\DeclareMathAlphabet\eusm{U}{eus}{m}{n}
\def\makebb#1{\expandafter\def\csname bb#1\endcsname{{\mathbb{#1}}}\ignorespaces}
\def\makerm#1{\expandafter\def\csname rm#1\endcsname{{\rm #1}}\ignorespaces}
\def\makebf#1{\expandafter\def\csname bf#1\endcsname{{\bf #1}}\ignorespaces}
\def\makegr#1{\expandafter\def\csname gr#1\endcsname{{\mathfrak{#1}}}\ignorespaces}
\def\makescr#1{\expandafter\def\csname scr#1\endcsname{{\mathscr{#1}}}\ignorespaces}
\def\makecal#1{\expandafter\def\csname cal#1\endcsname{{\cal #1}}\ignorespaces}
\def\makeudl#1{\expandafter\def\csname udl#1\endcsname{{\underline{#1}}}\ignorespaces}
\def\doLetters#1{%
  #1A #1B #1C #1D #1E #1F #1G #1H #1I #1J #1K #1L #1M
  #1N #1O #1P #1Q #1R #1S #1T #1U #1V #1W #1X #1Y #1Z}
\def\doletters#1{%
  #1a #1b #1c #1d #1e #1f #1g #1h #1i #1j #1k #1l #1m
  #1n #1o #1p #1q #1r #1s #1t #1u #1v #1w #1x #1y #1z}
\doLetters\makebb\doLetters\makecal\doLetters\makerm\doLetters\makebf\doLetters\makescr
\doletters\makerm \doletters\makebf \doLetters\makegr \doletters\makegr \doLetters\makeudl \doletters\makeudl
    \def\setminus{\smallsetminus}
   
\def\ringO{{\scrO}}

\makeatletter   \newdimen\mina@@\mina@@=18pt
\newcommand{\xrtarw}[2][]{\mathrel{\mathop{\,\setbox\z@\vbox{\m@th
  \hbox{$\scriptstyle\;{#1}\;\;$}\hbox{$\m@th\scriptstyle\;{#2}\;\;$}}%
  \hbox to\ifdim\wd\z@>\mina@@\wd\z@\else\mina@@\fi{\rightarrowfill@
  \displaystyle}\,}\limits^{#2}\@ifnotempty{#1}{_{#1}}}}
\newcommand{\xltarw}[2][]{\mathrel{\mathop{\,\setbox\z@\vbox{\m@th
  \hbox{$\scriptstyle\;\;{#1}\;$}\hbox{$\m@th\scriptstyle\;\;{#2}\;\;$}}%
  \hbox to\ifdim\wd\z@>\mina@@\wd\z@\else\mina@@\fi{\leftarrowfill@
  \displaystyle}\,}\limits^{#2}\@ifnotempty{#1}{_{#1}}}}
\makeatother
\def\XYmatrix{\xymatrix@M=5pt} 
\def\ncmd{\newcommand}
\ncmd{\xysubset}[1][r]{\ar@<-2.5pt>@{^(-}[#1]\ar@<2.5pt>@{_(-}[#1]}
\ncmd{\XYmatrixc}[1]{\vcenter{\XYmatrix{#1}}}
\ncmd{\xyto}[1][r]{\ar@{->}[#1]}      \ncmd{\xyinj}[1][r]{\ar@{^(->}[#1]}
\ncmd{\xysurj}[1][r]{\ar@{->>}[#1]}   \ncmd{\xyline}[1][r]{\ar@{-}[#1]}
\ncmd{\xydotsto}[1][r]{\ar@{.>}[#1]}  \ncmd{\xydots}[1][r]{\ar@{.}[#1]}
\ncmd{\xyleadsto}[1][r]{\ar@{~>}[#1]} \ncmd{\xyeq}[1][r]{\ar@{=}[#1]}
\ncmd{\xyequal}[1][r]{\ar@{=}[#1]}    \ncmd{\xyequals}[1][r]{\ar@{=}[#1]}
\ncmd{\xymapsto}[1][r]{\ar@{|->}[#1]}\ncmd{\xyimplies}[1][r]{\ar@{=>}[#1]}
\ncmd{\xytofrom}[1][r]{\ar@{<->}[#1]} 
  
\def\XYTOTO[#1]^#2_#3{\xyto[#1]<0.5ex>^{#2}\xyto[#1]<-0.5ex>_{#3}}
\ncmd{\xytoto}[1][r]{\XYTOTO[#1]}
\def\beginmat{\begin{pmatrix}}\def\endmat{\end{pmatrix}}
\def\pmat#1]{{\def\beginmat{\begin{pmatrix}}\def\endmat{\end{pmatrix}}\mat#1]}}
\def\bmat#1]{{\def\beginmat{\begin{bmatrix}}\def\endmat{\end{bmatrix}}\mat#1]}}
\def\Bmat#1]{{\def\beginmat{\begin{Bmatrix}}\def\endmat{\end{Bmatrix}}\mat#1]}}
\def\vmat#1]{{\def\beginmat{\begin{vmatrix}}\def\endmat{\end{vmatrix}}\mat#1]}}
\def\Vmat#1]{{\def\beginmat{\begin{Vmatrix}}\def\endmat{\end{Vmatrix}}\mat#1]}}
\def\smat#1]{{\def\beginmat{\begin{smallmatrix}}%
  \def\endmat{\end{smallmatrix}}\left(\mat#1]\right)}}
\def\mat#1#2]{\ifcase#1\or \matA#2]\or \matAA#2]\or \matAAA#2]\fi}
\def\matA  #1#2]{\ifcase#1\or \matAB  #2]\or \matABB  #2]\or \matABBB  #2]\fi}
\def\matAA #1#2]{\ifcase#1\or \matAAB #2]\or \matAABB #2]\or \matAABBB #2]\fi}
\def\matAAA#1#2]{\ifcase#1\or \matAAAB#2]\or \matAAABB#2]\or \matAAABBB#2]\fi}
\def\matAB[#1]{\beginmat#1\endmat}
\def\matABB[#1,#2]{\beginmat#1&#2\endmat}
\def\matABBB[#1,#2,#3]{\beginmat#1&#2&#3\endmat}
\def\matAAB[#1;#2]{\beginmat#1\\#2\endmat}
\def\matAABB[#1,#2;#3,#4]{\beginmat#1&#2\\#3&#4\endmat}
\def\matAABBB[#1,#2,#3;#4,#5,#6]{\beginmat
   #1&#2&#3\\#4&#5&#6\endmat}
\def\matAAAB[#1;#2;#3]{\beginmat#1\\#2\\#3\endmat}
\def\matAAABB[#1,#2;#3,#4;#5,#6]{\beginmat
   #1&#2\\#3&#4\\#5&#6\endmat}
\def\matAAABBB[#1,#2,#3;#4,#5,#6;#7,#8,#9]{\beginmat
   #1&#2&#3\\#4&#5&#6\\#7&#8&#9\endmat}
\def\beginalignorgather#1#2\endalignorgather{%
  \ifx#1!\beginaorgnostar#2\endaorgnostar\else\beginaorgstar#1#2\endaorgstar\fi}
\def\beginaorgstar#1#2\endaorgstar{%
  \ifx#1@\begin{align*}#2\end{align*}\else\begin{gather*}#1#2\end{gather*}\fi}
\def\beginaorgnostar#1#2\endaorgnostar{%
  \ifx#1@\begin{align}#2\end{align}\else\begin{gather}#1#2\end{gather}\fi}
\def\[#1\]{\beginalignorgather#1\endalignorgather}
\def\simto{\xrtarw{\sim}}

\def\dbltag#1#2{\tag*{\hbox to 0pt{\hbox to \hsize
  {\hfil#2}\hss}#1}}

\newcommand{\lowsim}{\smash{\hbox{\lower2.5pt
  \hbox{\(\scriptstyle\sim\)}}}}
\def\leq{\leqslant}
\def\geq{\geqslant}
\def\ge{\geq}
\def\le{\leq}

\newtheorem{theorem}{Theorem}[section]
\newtheorem{lemma}[theorem]{Lemma}
\newtheorem{corollary}[theorem]{Corollary}
\newtheorem{proposition}[theorem]{Proposition}

\newtheorem*{lemma*}{Lemma}
\newtheorem*{theorem*}{Main Theorem}

\theoremstyle{definition}

\newtheorem{example}[subsection]{Example}
\theoremstyle{remark}
\newtheorem*{remark}{Remark}
\numberwithin{equation}{section}
\def\ad{{\rm ad}}
\def\sc{{\rm sc}}
\def\gl{{\rm gl}}
\def\st{{\rm st}}

\usepackage{comment}
\usepackage{hyperref}
\hypersetup{
  hidelinks,
  pdftitle={Zeta functions of PGLn over non-Archimedean local fields},
  pdfauthor={Ming-Hsuan Kang and Jiu-Kang Yu}
}
\begin{document}
\title[zeta functions]{Zeta functions of PGL$_n$ over non-Archimedean local fields}
\author{Ming-Hsuan Kang}
\address{National Yang Ming Chiao Tung University, Department of Mathematics,
Hsinchu, Taiwan 300} 
\email{mhkang@math.nctu.edu.tw}
\author{Jiu-Kang Yu}
\address{The Institute of Mathematical Sciences, The Chinese University
  of Hong Kong\\Shatin, New Territories, Hong Kong}
\email{jkyu@ims.cuhk.edu.hk}
\begin{abstract}
Let \(\scrB\) be the Bruhat--Tits building of
\(\mathrm{PGL}_n(F)\), where \(F\) is a non-Archimedean local field.  We
introduce geometric \(k\)-geodesics in \(\scrB\) by means of CAT(0)
convexity and combinatorial \(k\)-geodesics by a local successor relation on
pointed \(k\)-facets.  We prove that the two notions coincide.  This allows
us to use the local combinatorial definition on quotients
\(\Gamma\backslash\scrB\), without referring to the universal covering.
When \(\Gamma\) is discrete, torsion-free, cocompact, and type-preserving,
the primitive closed \(k\)-geodesics define zeta functions \(Z_k\) and their
\(\epsilon\)-twisted variants \(Z_k^\epsilon\).  Our main result identifies
an alternating product of these zeta functions with the unramified
\(L\)-function of \(L^2(\Gamma\backslash \mathrm{PGL}_n(F))\):
\[
(1-u^n)^{\chi(\Gamma\backslash\scrB)}
L(\Gamma,q^{(n-1)/2}u)
=
\prod_{k=1}^{n-1} Z_k^\epsilon(\Gamma\backslash\scrB,u)^{(-1)^{k+1}}.
\]
This gives a uniform Ihara-type identity for all \(\mathrm{PGL}_n\).  We also
extend the construction and the identity to \(\mathrm{PGL}_n(D)\), where
\(D\) is a central division algebra over \(F\); in that setting the residue
parameter is \(Q=|\ringO_D/\mathfrak p_D|\).
\end{abstract}
\maketitle
\def \X {\scrB_\Gamma}
\def \oo {\mathrm{O}}
\def \I {\mathrm{I}}
\def \G {\mathrm{G}}
\def \Q {\mathbb{Q}}
\def \scrC{\mathcal{C}}
\def \C {\mathbb{C}}

\section{Introduction}
\enlargethispage{3pt}

\noindent\textbf{Ihara's zeta function.}
In 1966, Ihara introduced a zeta function for torsion-free discrete subgroups
of \(\mathrm{PGL}_2\) over a \(p\)-adic field~\cite{Ih}.  When the subgroup
\(\Gamma\) is cocompact, it acts freely on the Bruhat--Tits tree and the
quotient is a finite graph \(X_\Gamma\).  The construction was a \(p\)-adic
analogue of Selberg's zeta function for compact quotients of the upper
half-plane.  Its essential geometric feature is the same: it counts primitive
closed geodesics.

For a graph, a geodesic has an equivalent combinatorial description as a
backtrackless, tailless path.  This makes the Euler product accessible to
linear algebra.  If \(F\) has residue field of cardinality \(q\), the
Bruhat--Tits tree is \((q+1)\)-regular, and Ihara's identity may be written
\begin{equation}\label{eq:Ihara}
 Z(X_\Gamma,u)
 =\frac{(1-u^2)^{\chi(X_\Gamma)}}{\det(I-Au+qu^2I)}
 =(1-u^2)^{\chi(X_\Gamma)}L(\Gamma,q^{1/2}u).
\end{equation}
Here \(A\) is both the graph adjacency operator and the spherical Hecke
operator, while \(L(\Gamma,u)\) is the unramified \(L\)-function attached to
the unramified spectrum of \(L^2(\Gamma\backslash\mathrm{PGL}_2(F))\); see
also~\cite{Ha}.  Thus~\eqref{eq:Ihara} connects a zeta function defined by
geodesics to an \(L\)-function from the representation theory of
\(\mathrm{PGL}_2(F)\).  This circle of ideas became part of the arithmetic
theory of regular graphs, notably the construction of Ramanujan
graphs~\cite{LPS}, and later of Ramanujan complexes~\cite{LSV}.

\medskip
\noindent\textbf{The higher-dimensional problem.}
Beginning with Deitmar--Hoffman's work on \(\mathrm{PGL}_3\)~\cite{DH},
Ihara's theory was extended from trees to two-dimensional buildings.  Precise
zeta identities were subsequently obtained for quotients of the buildings of
\(\mathrm{PGL}_3\)~\cites{KL,KLW} and \(\mathrm{PGSp}_4\)~\cite{FLW}.
As in rank one, these identities connect \(L\)-functions and zeta functions
that count geodesics.  A genuinely new feature, however, is that one needs
several kinds of zeta functions.  In the \(\mathrm{PGL}_3\) identity of
\cite{KL}, for example, edge and chamber zeta functions occur together.

The geometry behind those two zeta functions suggests how to proceed in
arbitrary dimension.  The paths counted by the edge zeta function are
geodesic curves on the building quotient paved by edges (1-simplices), and
their lifts to the building are straight line segments paved by edges.  The
chamber geodesics are even more revealing after lifting them to the building:
in an apartment they form straight strips paved by chambers (2-simplices).
Thus the apparently different edge and chamber constructions are instances of
one geometric phenomenon.  Finding a formulation that captures all such types
simultaneously is a principal obstacle to a uniform arbitrary-rank identity.

This paper resolves that problem for the entire family
\(\mathrm{PGL}_n(F)\), \(n\ge2\).  Since its building has dimension \(n-1\),
this is a single theory encompassing geometry of arbitrarily high dimension.
To our knowledge, it is the first Ihara-type identity uniform in \(n\) for
this family.

\medskip
\noindent\textbf{Geometric and combinatorial \(k\)-geodesics.}
The conceptual core of the paper is the identification of two notions on the
building \(\scrB\).  A \(k\)-geodesic should be viewed as a \(k\)-dimensional
subcomplex of \(\scrB\), paved by \(k\)-facets; except when \(k=1\), it is not
a curve.  We introduce two descriptions of this object, one geometric and one
combinatorial.  A geometric \(k\)-geodesic is defined by CAT(0)
convexity: the convex hull of any two of its \(k\)-facets is exactly the union
of the facets between them.  In its standard bi-infinite realization, such a
geodesic is a straight strip in an apartment, isometric to
\(S\times\mathbb R\) for a Euclidean \((k-1)\)-simplex \(S\), and paved by
\(k\)-facets.

A combinatorial \(k\)-geodesic is a sequence of pointed \(k\)-facets linked
by a successor relation.  Crucially, this relation is purely local.  Two
consecutive facets meet along a common \((k-1)\)-facet, and the relevant
directions are opposite in its link; the link is a spherical building.

Proving that these descriptions agree is substantial: Sections~4--5 are
devoted to it.  Every successor chain underlies a geometric \(k\)-geodesic,
and every geometric \(k\)-geodesic has a compatible pointing.  For \(k\ge2\)
the compatible orientation is unique; for \(k=1\) the two orientations give
the two directed paths.  After proving this equivalence, we simply call the
common object a \emph{\(k\)-geodesic}.  On a quotient
\(\scrB_\Gamma=\Gamma\backslash\scrB\), we use the local combinatorial
description as the definition.  In particular, when \(\Gamma\) is trivial,
this definition is exactly the combinatorial notion on the building.

\medskip
\noindent\textbf{Zeta functions and the main theorem.}
Assume now that \(\Gamma\) is also cocompact, so that \(\scrB_\Gamma\) is
finite.  For each \(1\le k\le n-1\), the primitive closed directed
\(k\)-geodesics \(\scrC\) define
\[
 Z_k(\scrB_\Gamma,u)
 =\prod_{[\scrC]}(1-u^{l_A(\scrC)})^{-1}.
\]
Here \(l_A\) is the algebraic length; the geometric length \(l_G\) records
the number of facets.  The sign
\(
 \epsilon(\scrC)=(-1)^{(k+1)l_G(\scrC)}
\)
gives the twisted zeta function
\[
 Z_k^\epsilon(\scrB_\Gamma,u)
 =\prod_{[\scrC]}
  (1-\epsilon(\scrC)u^{l_A(\scrC)})^{-1}.
\]
The different values of \(k\) organize, rather than eliminate, the several
types of geodesics that first appeared in dimension two.

Our main theorem is the following.
\begin{theorem*}
Let \(F\) be a non-Archimedean local field with residue field of cardinality
\(q\), let \(\scrB\) be the Bruhat--Tits building of
\(\mathrm{PGL}_n(F)\), and let \(\Gamma\) be a discrete torsion-free
cocompact subgroup that preserves vertex types.  Then
\[
(1-u^n)^{\chi(\scrB_\Gamma)}
L(\Gamma,q^{(n-1)/2}u)
=
\prod_{k=1}^{n-1}
Z_k^\epsilon(\scrB_\Gamma,u)^{(-1)^{k+1}}.
\]
Here \(L(\Gamma,u)\) is the unramified standard \(L\)-function attached to the
unramified spectrum of
\(L^2(\Gamma\backslash\mathrm{PGL}_n(F))\), and
\(\chi(\scrB_\Gamma)\) is the Euler characteristic.
\end{theorem*}
For \(n=2\), the two lengths agree and the theorem is precisely
Ihara's identity~\eqref{eq:Ihara}.  For \(n=3\), after matching conventions,
it gives the main identity of~\cite{KL}, with the \(k=1\) and \(k=2\) factors
corresponding to the edge and chamber contributions.  There is a small
difference in scope: the quotient in~\cite{KL} is treated as a simplicial
complex, whereas here it is endowed with its natural \(\Delta\)-complex
structure.  Thus our formulation also applies when distinct facets of the
quotient have the same set of vertices.  The theorem therefore both recovers
the known low-dimensional cases and explains why all dimensions of facets
enter in arbitrary rank.

\medskip
\noindent\textbf{Central division algebras.}
The method is insensitive to the commutativity of the coefficient field.  Let
\(D\) be a central division algebra of degree \(d\) over \(F\), let
\(\ringO_D\) be its maximal order, and put
\[
Q=\#(\ringO_D/\mathfrak p_D)=q^d.
\]
The building of \(\mathrm{PGL}_n(D)\) is again of type
\(\widetilde A_{n-1}\).  Replacing \(F\)-lattices by right
\(\ringO_D\)-lattices leaves the geometric arguments unchanged, while every
finite vector-space calculation takes place over the residue field
\(\ringO_D/\mathfrak p_D\).  Consequently all the geodesic constructions and
the main identity extend to \(\mathrm{PGL}_n(D)\), with \(q\) replaced by
\(Q\) and with the degree-\(n\) relative spherical \(L\)-function.  The
precise statement and proof are given in Section~9.

\medskip
\noindent\textbf{The proof.}
The \(\mathrm{PGL}_3\) identity just mentioned has two published proofs: the
combinatorial proof of~\cite{KL} and the representation-theoretic proof
of~\cite{KLW}.  Our proof is different from both and, upon specialization,
gives another proof of that identity.  It is inspired instead by Hoffman's
reformulation of Bass's proof of the Ihara identity in terms of torsion of
complexes~\cite{Hof}.

The higher-rank argument is necessarily more elaborate.  The Euler products
are first expressed as determinants of successor operators on pointed
\(k\)-facets.  These spaces, for all \(k\), are assembled into a cochain
complex.  A family of Laplacian-type endomorphisms then compares the
alternating product of successor determinants with the spherical Hecke
determinant for the unramified \(L\)-function.  The principal technical step
is a M\"obius-inversion calculation of these endomorphisms.  This strategy
retains the cohomological idea behind Hoffman's proof while accommodating all
facet dimensions and all possible algebraic lengths.

Other related developments include Artin \(L\)-functions for
\(\mathrm{PGL}_3\) quotients~\cite{KL2}, zeta and \(L\)-functions for
apartments and buildings~\cite{KLW2}, geometric zeta functions for
higher-rank \(p\)-adic groups~\cite{DKgeo}, Bass-type zeta functions for
building lattices~\cite{DKM}, and twisted Poincar\'e-series identities for
rank-two buildings~\cite{KM}.  See also~\cite{DK} and~\cite{Sa} for related
constructions.

\medskip
\noindent\textbf{Organization of the paper.}
Section~2 reviews Bruhat--Tits theory, the simplicial structure of \(\scrB\),
and the geometric and algebraic length functions.
Section~3 translates this structure into the lattice model used in the proofs.
Section~4 studies geometric \(k\)-geodesics, exhibits the standard straight
strip, and proves the geometric normal form.  Section~5 introduces
combinatorial \(k\)-geodesics and proves their equivalence with the geometric
ones.  Section~6 passes to quotient \(\Delta\)-complexes, defines
\(k\)-geodesics there by the local successor relation, and then imposes
finiteness to define the zeta functions and state the main theorem.  Section~7
proves the main theorem modulo a determinant computation, whose technical
proof occupies Section~8.  Section~9 extends the theory to central division
algebras over \(F\).

\section{Bruhat--Tits Theory}

This section fixes the language used throughout the paper.  We recall the type
of a facet, the ordered type of a pointed facet, opposite neighboring facets,
and the successor relation.  These notions are intrinsic to the building, while
Section~3 gives their concrete realization in terms of lattice chains.

\subsection{Non-Archimedean Local Fields}

Throughout the paper, let \(n\ge2\), and let \(F\) be a non-Archimedean local
field with ring of integers \(\ringO\), uniformizer \(\pi\), and residue field
\(\kappa=\ringO/\pi\ringO\) of cardinality \(q\).

\subsection{Ordered Partitions}

Let \(n\) and \(k\) be positive integers. An \textit{ordered partition of \(n\) into \(k\) parts} is a \(k\)-tuple \((\lambda_0, \ldots, \lambda_{k-1})\) of positive integers such that \(\lambda_0 + \cdots + \lambda_{k-1} = n\). We always regard such a \(k\)-tuple as a family indexed by the set \(\{0, 1, \ldots, k-1\}\), which we also identify with \(\mathbb{Z}/k\mathbb{Z}\).

\subsection{Circular Partitions}

Let \(\scrP_k(n)\) be the set of ordered partitions of \(n\) into \(k\) parts. The group \(\mathbb{Z}/k\mathbb{Z}\) acts on \(\scrP_k(n)\) by cyclic rotation, and a \(\mathbb{Z}/k\mathbb{Z}\)-orbit in \(\scrP_k(n)\) is called a \textit{circular partition of \(n\) into \(k\) parts}. We denote the \(\mathbb{Z}/k\mathbb{Z}\)-orbit of \((\lambda_0, \ldots, \lambda_{k-1})\) by \((\!(\lambda_0, \ldots, \lambda_{k-1})\!)\).

\subsection{The Building}

The groups \(\GL_n(F)\), \(\SL_n(F)\), and \(\PGL_n(F)\) share the same (reduced) Bruhat--Tits building \(\scrB\); see \cite{BT}, \cite{tits}, and \cite{AB}. We denote their groups of \(F\)-points by \(G_\gl\), \(G_\sc\), and \(G_\ad\), respectively. Let \(G_\dag\) be any of these groups (\(\dag = \gl, \ad, \sc\)).

\subsection{Affine Dynkin Labels}

Let \(\Delta_{\mathrm{aff}}\) be the set of nodes of the \textit{affine} Dynkin diagram of \(G_\dag\), and let \(\Xi_\dag\) be Bruhat--Tits' \(\Xi\) group of \(G_\dag\). The diagram \(\Delta_{\mathrm{aff}}\) is a circular diagram with \(n\) nodes, independent of \(\dag\). The group \(\Xi_\dag\) is trivial for \(\dag = \sc\), and \(\Xi_\dag = \mathbb{Z}/n\mathbb{Z}\) for \(\dag = \gl\) or \(\ad\). The group \(\Xi_\dag\) acts naturally on \(\Delta_{\mathrm{aff}}\), and \(\Delta_{\mathrm{aff}}\) is a principal homogeneous space for \(\Xi_\dag\) when \(\dag = \gl\) or \(\ad\).

The circular diagram \(\Delta_{\mathrm{aff}}\) has two orientations when \(n \geq 3\), which are permuted by the outer automorphism of \(G_\ad\). We will fix an orientation once and for all. Choosing such an orientation is equivalent to selecting one of the two \(n\)-dimensional fundamental representations of \(G_\sc\) as the standard representation. The choice of orientation also determines an action of \(\mathbb{Z}/n\mathbb{Z}\) on \(\Delta_{\mathrm{aff}}\), written as \((i, s) \mapsto i + s\), such that the vertices on \(\Delta_{\mathrm{aff}}\) connected to \(s\) are \(\pm 1 + s\).

\subsection{The Circular Partition Associated with a Facet}\label{bt-facet-type}

Let \(C\) be a \(k\)-facet of \(\scrB\), and let \(P_C\) be its associated
parahoric subgroup. The set \(S\) of vertices of \(C\) corresponds to a subset
\(T\) of \(\Delta_{\mathrm{aff}}\), well-defined modulo the action of
\(\Xi_\dag\), with \(\#T = k + 1\). The Dynkin diagram of the maximal reductive
quotient of \(P_C\) is obtained by removing the nodes in \(T\) from the diagram
\(\Delta_{\mathrm{aff}}\). It is then of type
\(A_{\lambda_0 - 1}, \ldots, A_{\lambda_k - 1}\), such that
\(\lambda_0 + \cdots + \lambda_k = n\). Since this is a subdiagram of
\(\Delta_{\mathrm{aff}}\), it determines a well-defined circular partition
\((\!(\lambda_0, \ldots, \lambda_k)\!)\) of \(n\) into \(k + 1\) parts,
called the \textit{type} of \(C\).

In addition, the circular diagram structure provides a natural notion of \textit{successors} on \(S\) as follows. Suppose \(v, v' \in S\) correspond to \(i, i + s \in \Delta_{\mathrm{aff}}\), with \(1 \leq s \leq n\). We say that \(v'\) is the \textit{successor of \(v\) on \(C\)} if \(i + 1, \ldots, i + s - 1 \notin T\). Since \(T\) is finite and the diagram is circular, every \(v \in S\) has a unique successor.

\subsection{Pointed Facets and Associated Ordered Partitions}\label{bt-pointed-type}
We will also work with \textit{pointed facets} of \(\scrB\), which are pairs \((C, v)\), where \(C\) is a facet and \(v\) is a vertex of \(C\). Assume that \(v\) corresponds to \(c \in \Delta_{\mathrm{aff}}\), and let
\[
T = \{c, c + \lambda_0, c + \lambda_0 + \lambda_1, \ldots, c + \lambda_0 + \cdots + \lambda_{k-1}\}
\]
such that \((\lambda_0, \ldots, \lambda_k)\) is an ordered partition of \(n\) into \(k + 1\) parts, canonically associated with \((C, v)\). We refer to \((\lambda_0, \ldots, \lambda_k)\) as \textit{the type of \((C, v)\)}, to \(v\) as \textit{the \(0\)-th vertex of \((C, v)\)}, and to the vertex \(v_1\) corresponding to \(c + \lambda_0\) as \textit{the first vertex of \((C, v)\)}, and so on.

Equivalently, a pointed \(k\)-facet \((C,v)\) may be represented by the ordered
sequence of its vertices
\[
(C,v)\;=\;(v_0,v_1,\ldots,v_k),
\]
where \(v_0=v\). After the chosen orientation identifies \(\Delta_{\mathrm{aff}}\) with
\(\mathbb Z/n\mathbb Z\), the successive cyclic differences
\[
v_{i+1}-v_i=\lambda_i\quad (0\leq i<k),\qquad
v_0-v_k=\lambda_k
\]
encode the ordered partition \((\lambda_0,\ldots,\lambda_k)\).
This ordered vertex sequence provides a convenient combinatorial model for
pointed facets and will be used throughout the sequel.

\subsection{Geometric and Algebraic Lengths}\label{length}

Let \(\dot C=(C,v)\) be a pointed \(k\)-facet of type
\((\lambda_0,\ldots,\lambda_k)\).  We assign to \(\dot C\) the
\emph{one-step geometric length} and the \emph{algebraic length}
\[
l_G(\dot C)=1,
\qquad
l_A(\dot C)=\lambda_0.
\]
Thus geometric length records one step, whereas algebraic length weights that
step by the first part of its ordered type.  These definitions depend only on
the typed pointed facet and therefore also apply to the quotient
\(\Delta\)-complex introduced in Section~\ref{quot-cx}.

\subsection{Opposite Neighboring Facets}\label{bt-opp}

Assume \(k\geq 1\).  Let \(C\) and \(C'\) be two distinct \(k\)-facets of
\(\scrB\) sharing a common \((k-1)\)-facet \(C''\), which is necessarily
unique.  The two coface incidences
\[
C''\longrightarrow C,
\qquad
C''\longrightarrow C'
\]
determine two vertices in the spherical building
\(\operatorname{Lk}_{\scrB}(C'')\); see
\cite{BT}*{5.1.32} and \cite{tits}*{3.5.4}.  We call \(C\) and \(C'\)
\textit{opposite neighboring facets} if these two vertices are opposite in this
link.  We use opposition in the sense of \cite{AB}*{4.68 and 4.72}, the latter
being the extension from chambers to arbitrary simplices.

Equivalently, in any apartment \(A\) of \(\scrB\) containing \(C\) and \(C'\),
every wall in \(A\) containing \(C''\), but not \(C\) or \(C'\), separates
\(C\) and \(C'\).

\subsection{The Successor Relation}\label{bt-successor-relation}

Let \((C,v)\) be a pointed \(k\)-facet of type
\((\lambda_0,\ldots,\lambda_k)\), and write \((C,v)\) in ordered form as
\[
(C,v)=(v_0,v_1,\ldots,v_k),
\]
where \(v_0=v\) and the successive differences encode the ordered partition
\((\lambda_0,\ldots,\lambda_k)\).

Let \(C''\) be the codimension-\(1\) facet of \(C\) opposite the vertex \(v_0\),
and let \(C'\) be any \(k\)-facet such that the vertices of
\(\operatorname{Lk}_{\scrB}(C'')\) determined by the incidences
\(C''\to C\) and \(C''\to C'\) are opposite.
Let \(v'=v_1\), the first vertex of \((C,v)\).
Then \((C',v')\) is a pointed \(k\)-facet whose type is
\[
(\lambda_1,\lambda_2,\ldots,\lambda_{k-1},\lambda_0,\lambda_k).
\]
We call \((C',v')\) a \textit{successor} of \((C,v)\).

Equivalently, when both pointed facets are written as ordered sequences of
vertices,
\[
(C,v)=(v_0,v_1,\ldots,v_k),
\qquad
(C',v')=(v_1,v_2,\ldots,v_k,v_{k+1}),
\]
the successor relation requires that the last \(k\) vertices of \((C,v)\)
coincide with the first \(k\) vertices of \((C',v')\), and that the underlying
facets \(C\) and \(C'\) are opposite neighboring facets across this common
\((k-1)\)-face.
Thus the successor relation is an intrinsic combinatorial operation on pointed
facets: it shifts the ordered vertex sequence forward and crosses the common
\((k-1)\)-face to an opposite neighboring facet.  The next section translates
this intrinsic description into lattice chains, where algebraic length and the
successor operation can be computed explicitly.

\section{Lattices}

We now pass from the intrinsic Bruhat--Tits description to the lattice model.
This concrete model will be used in all later arguments: it realizes facets as
chains of lattices, identifies their ordered types, computes algebraic length,
and gives an explicit form of the successor relation.

\subsection{The Lattice Model of the Building} 
The Bruhat--Tits building \(\scrB\) of \(G_\dag\) is a simplicial complex that can be described as follows. Throughout, we use \(k\)-simplex and \(k\)-facet interchangeably. For any lattice \(L\) in \(F^n\), we define its homothety class as
\[
[L] := \{\alpha L : \alpha \in F^\times\}.
\]

The set of vertices \(\scrB_0\) on \(\scrB\) consists of all such homothety classes of lattices in \(F^n\). A set of vertices \(\{\scrL_0, \ldots, \scrL_k\}\) forms a \(k\)-simplex (also called a \(k\)-facet) if there exist representatives \(L_i \in \scrL_i\) for each \(i\) and a permutation \(\sigma\) of \(\{0, \ldots, k\}\) such that
\[
L_{\sigma(0)} \supsetneq L_{\sigma(1)} \supsetneq \cdots \supsetneq L_{\sigma(k)} \supsetneq \pi L_{\sigma(0)}.
\]
The natural action of \(G_\dag\) on \(\scrB_0\) preserves this simplicial structure.

\subsection{Vertex Types in the Lattice Model} 
We refer to \(\ringO_F^n\) as the \textit{standard lattice in \(F^n\)} and denote it by \(L_\st\). For any lattice \(L\), there exists a lattice \(L'\) such that \(L' \subset L_\st\) and \(L' \subset L\). The number
\[
\widetilde{\imath}(L)
:=
\length(L_\st/L')-\length(L/L')
\]
depends only on \(L\), not on \(L'\), and its image
\[
i([L]):=\widetilde{\imath}(L)+n\mathbb Z
\quad\text{in}\quad \mathbb Z/n\mathbb Z
\]
depends only on the homothety class \([L]\).

The fibers of the surjection \(i : \scrB_0 \to \mathbb{Z}/n\mathbb{Z}\) are exactly the \(G_\sc\)-orbits, allowing us to identify \(\Delta_{\mathrm{aff}} := G_\sc \backslash \scrB_0\) with \(\mathbb{Z}/n\mathbb{Z}\).

The group \(\Xi_\dag\) is defined as the image of \(G_\dag\) under the map to \(\Aut(\Delta_{\mathrm{aff}})\), capturing the automorphism structure of the reduced building. Therefore, \(\Xi_\sc\) is trivial. The simple formula \(i(g\scrL) = \ord(\det g) + i(\scrL)\) for \(g \in G_\gl\) shows that \(\Xi_\ad = \Xi_\gl\) are cyclic groups of order \(n\), and \(\Delta_{\mathrm{aff}}\) is a principal homogeneous space for \(\Xi_\ad\) and \(\Xi_\gl\).

\subsection{The Circular Partition Associated with a Facet}\label{lattice-facet-type}
Let \(C\) be a \(k\)-facet with vertex set
\[
S = \{[L_0], \ldots, [L_k]\}.
\]
We may and do assume
\[
L_0 \supsetneq L_1 \supsetneq \cdots \supsetneq L_k \supsetneq \pi L_0. \tag{$\ast$}
\]
Define
\(\lambda_i=\dim_\kappa(L_i/L_{i+1})\) for \(0\le i\le k-1\), and
\(\lambda_k=\dim_\kappa(L_k/\pi L_0)\). Then
\((\lambda_0, \ldots, \lambda_k)\) is an ordered partition of \(n\) into
\(k + 1\) parts. The circular partition
\((\!(\lambda_0, \ldots, \lambda_k)\!)\) is independent of the choice of
\(L_0, \ldots, L_k\) to represent \(S\) and is a \(G_\dag\)-invariant
associated with \(C\).

With this description of \(S\), we say that \([L_{i+1}]\) is the \textit{successor of \([L_i]\) on \(C\)} for each \(i \in \mathbb{Z}/(k + 1)\mathbb{Z}\).

\subsection{Pointed Facets and Associated Ordered Partitions}\label{lattice-pointed-type}
Let \((C, v)\) be a pointed \(k\)-facet in the building \(\scrB\), with \(v = [L_0]\) as its vertex. There exist unique lattices \(L_1, \ldots, L_k\) satisfying condition \((\ast)\), such that the vertices of \(C\) are represented by \([L_0], \ldots, [L_k]\). Define \(\lambda_i\) as above. The ordered partition \((\lambda_0, \ldots, \lambda_k)\) is a \(G_\dag\)-invariant associated with \((C, v)\), reflecting both the geometric and algebraic structures of the facet.

We denote \((C,v)\) by \([L_0,L_1,\ldots,L_k]\).  In this model, the
algebraic length of Section~\ref{length} is computed by
\[
l_A(C,v)=\lambda_0=\dim_\kappa(L_0/L_1).
\]

\subsection{Opposite Neighboring Facets and Successors}\label{successor}
Consider two distinct \(k\)-facets \(C\) and \(C'\) in \(\scrB\) that share a common \((k - 1)\)-subfacet \(C''\). Let \(v\) and \(\hat{v}\) be the unique vertices of \(C\) and \(C'\), respectively, that are not contained in \(C''\). We represent the pointed facets \((C, v)\) and \((C', \hat{v})\) by the lattice sequences \([L_0, \ldots, L_k]\) and \([L_0', \ldots, L_k']\), respectively, with the assumption that \(L_i = L_i'\) for \(i = 1, \ldots, k\).

The facets \(C\) and \(C'\) are called \textit{opposite neighboring facets} if the quotient \(\pi^{-1} L_k / L_1\) decomposes as the direct sum of \(L_0 / L_1\) and \(L_0' / L_1\). Equivalently, this condition is satisfied if \(L_0 \cap L_0' = L_1\) and \(L_0 + L_0' = \pi^{-1} L_k\). This lattice-theoretic characterization is consistent with the geometric notion of opposite facets discussed in Section~\ref{bt-opp}; see also \cite{AB}*{4.78}.

Next, we refine this concept by introducing the \textit{successor} relation between pointed facets associated with opposite neighboring facets. Let \(v' = [L_1'] = [L_1]\), so that the pointed facet \((C', v')\) is represented by the lattice sequence \([L_1', L_2', \ldots, L_k', \pi L_0'] = [L_1, L_2, \ldots, L_k, \pi L_0']\). We refer to \((C', v')\) as a \textit{successor} of \((C, v)\).

In summary, a pointed \(k\)-facet is a successor of another pointed \(k\)-facet if they are opposite neighboring facets, and the first \(k\) terms of the lattice sequence of the former equal the last \(k\) terms of the lattice sequence of the latter.

\section{Geometric \texorpdfstring{\(k\)-geodesics}{k-geodesics}}
\label{s:geometric-geodesic}

The next step is to isolate the geometric notion of a \(k\)-geodesic,
independent of pointings and successor operators.  We define it by CAT(0)
convexity: between any two facets in the sequence, the convex hull should be
exactly the union of the intervening facets.  The goal of this section is to
prove that such geodesics have a rigid normal form in an apartment.
We use the definitions of a CAT(0) space and of a convex subset given in
\cite{AB}*{11.3} and \cite{AB}*{11.27}, respectively.

Let \((C_0,\ldots,C_m)\) be a sequence of pairwise distinct \(k\)-facets in the
building \(\scrB\).  We say that \((C_0,\ldots,C_m)\) is a
\emph{geometric geodesic} if, for every \(0\le i\le j\le m\), the convex hull of
\(\overline{C}_i\cup \overline{C}_j\) in the CAT(0) realization of \(\scrB\) is
exactly
\[
\mathrm{conv}\bigl(\overline{C}_i\cup \overline{C}_j\bigr)
=
\overline{C}_i\;\cup\;\overline{C}_{i+1}\;\cup\;\cdots\;\cup\;\overline{C}_j.
\]
In particular, taking \(j=i+1\) shows that \(\overline{C}_i\cup\overline{C}_{i+1}\)
is convex for each \(i\); hence every consecutive pair \(C_i,C_{i+1}\) shares a
common \((k-1)\)-subfacet.

This convexity-based definition is global in nature.  To analyze it effectively,
we next study convexity properties of small configurations of facets, beginning
with the case of two \(k\)-facets sharing a common \((k-1)\)-subfacet.

\subsection{Convexity of Two Facets}
Let \(W\) be a finite Coxeter group. Assume that \(W\) arises from a root system \(\Phi\) in a Euclidean space \(V\), as described in \cite{Hu}. For simplicity, we assume \(\Phi\) spans \(V\) (i.e., \(W\) is essential). We will realize the Coxeter complex \(\Sigma(W)\) of \(W\) as in \cite{Hu}*{1.15}.

We begin with the spherical (finite) case, where convexity can be characterized
purely in terms of the local structure around a common subfacet.

\begin{lemma} \label{opp-convex-sph}
Let \(\Sigma(W)\) be the spherical Coxeter complex of a finite reflection group
\(W\) acting on the Euclidean space \(V\).  Let \(F, F'\) be \(k\)-facets of
\(\Sigma(W)\) sharing a common \((k-1)\)-subfacet \(F''\). The following are equivalent:
\begin{itemize}
    \item[\textup{(i)}] \(F, F'\) are opposite facets on the link of \(F''\).
    \item[\textup{(ii)}] \(\overline{F} \cup \overline{F}'\) is convex in \(V\).
    \item[\textup{(iii)}] There exists a \((k+1)\)-dimensional linear subspace of \(V\) containing \(\overline{F}\cup\overline{F'}\).
\end{itemize}
\end{lemma}

\begin{proof}
Let \(H=\operatorname{span}(F'')\).  In the quotient of
\(\operatorname{span}(F\cup F')\) by \(H\), the images of \(F\) and \(F'\)
are rays.  They are opposite in the link of \(F''\) precisely when these rays
are opposite.  Equivalently, \(F\cup F'\) is contained in a
\((k+1)\)-dimensional linear subspace.  This proves
\textup{(iii)}\(\Leftrightarrow\)\textup{(i)}.  If \(\overline F\cup\overline{F'}\) is convex,
the two transverse rays cannot be linearly independent, since the segment
joining interior points on the two rays would otherwise leave the union.
Thus \textup{(ii)}\(\Rightarrow\)\textup{(iii)}.

It remains to prove \textup{(i)}\(\Rightarrow\)\textup{(ii)}.  We first establish the projection
property needed for this implication.  Suppose that
\(\mathbb{R}_{>0}v_0,\ldots,\mathbb{R}_{>0}v_k\) are the vertices of \(F\),
and that \(F''\) has vertices
\(\mathbb{R}_{>0}v_1,\ldots,\mathbb{R}_{>0}v_k\).  We claim that the
orthogonal projection of \(v_0\) to \(H\) lies in the cone
\[
K=\operatorname{cone}(v_1,\ldots,v_k).
\]

First assume that \(k=\dim V-1\), so that \(F\) is a chamber.  Choose simple
roots \(\alpha_0,\ldots,\alpha_k\) dual to \(v_0,\ldots,v_k\).  We use the
standard inequality \((\alpha_i,\alpha_j)\le0\) for \(i\ne j\); see the
corollary in \cite{Hu}*{1.3}.  The projection in question is to
\(\alpha_0^\perp\), and is equal to
\[
p = v_0 - \frac{(v_0, \alpha_0)}{(\alpha_0, \alpha_0)}\alpha_0 = v_0 - \frac{1}{(\alpha_0, \alpha_0)}\alpha_0.
\]
The cone \(K\) is given by the conditions \((\alpha_0,-)=0\) and
\((\alpha_i,-)\ge0\) for \(i=1,\ldots,k\).  For these indices,
\[
(\alpha_i, p) = (\alpha_i, v_0) - \frac{1}{(\alpha_0, \alpha_0)}(\alpha_i, \alpha_0) = -\frac{(\alpha_i, \alpha_0)}{(\alpha_0, \alpha_0)} \geq 0
\]
and hence \(p\in K\).  When \(k<\dim V-1\), restrict the inner product to
the span of the facet.  The next lemma shows that the same off-diagonal sign
condition survives this restriction, so the preceding argument applies
unchanged.

Apply the projection property to both \(F\) and \(F'\).  Under condition (i),
their extra rays lie on opposite sides of \(H\) in a common
\((k+1)\)-dimensional subspace.  Write their generators as
\(v_0=p+an\) and \(v'_0=p'-a'n\), where \(a,a'>0\), \(n\perp H\), and
\(p,p'\in K\).  If \(x=bv_0+h\in\overline F\) and
\(x'=b'v'_0+h'\in\overline{F'}\), with \(b,b'\ge0\) and \(h,h'\in K\),
then a point on the segment \([x,x']\) lies in \(\overline F\) when its
\(n\)-coordinate is nonnegative and in \(\overline{F'}\) when that coordinate
is nonpositive: after subtracting the corresponding multiple of \(v_0\) or
\(v'_0\), the remaining component is a nonnegative combination of
\(h,h',p,p'\), hence lies in \(K\).  Thus every segment joining the two facets
is contained in their union, and \(\overline F\cup\overline{F'}\) is convex.
\end{proof}

The next lemma records the stability of the sign condition on inner products under passage to such subspaces.

\begin{lemma}
Let \(V\) be a finite-dimensional real inner-product space.  Let
\(\alpha_1, \ldots, \alpha_n\) be a basis of \(V\) such that
\((\alpha_i, \alpha_j) \leq 0\) for \(i \neq j\). Let \(v_1, \ldots, v_n\)
in \(V\) be the dual basis, i.e., \((\alpha_i, v_j) = \delta_{ij}\). Let
\(1 \leq k \leq n\) and let \(V'\) be the span of \(v_1, \ldots, v_k\).
Let \(\beta_1, \ldots, \beta_k\) be the basis of \(V'\) dual to
\(v_1, \ldots, v_k\). Then \((\beta_i, \beta_j) \leq 0\) for \(i \neq j\).
\end{lemma}

\begin{proof}
The result holds for \(k = n\) by assumption. We will show that it holds for \(k = n-1\). Then the general case follows by induction.

Since \(k = n-1\), \(V'\) is simply \(\alpha_n^\perp\). One easily checks that
\[
\beta_i = \alpha_i - \frac{(\alpha_i, \alpha_n)}{(\alpha_n, \alpha_n)}\alpha_n, \qquad i = 1, \ldots, k.
\]
Thus,
\[
(\beta_i, \beta_j) = (\alpha_i, \alpha_j) - \frac{(\alpha_i, \alpha_n)(\alpha_j, \alpha_n)}{(\alpha_n, \alpha_n)} \leq (\alpha_i, \alpha_j) \leq 0
\]
for \(i \neq j\).
\end{proof}

We now extend the same convexity criterion to the affine setting.

\begin{lemma} \label{opp-convex-aff}
Let \(A\) be an apartment of an irreducible affine Coxeter complex, and let
\(F,F'\) be \(k\)-facets sharing a common \((k-1)\)-subfacet \(F''\). The following are equivalent:
\begin{itemize}
    \item[\textup{(i)}] \(F, F'\) are opposite facets on the link of \(F''\).
    \item[\textup{(ii)}] \(\overline{F} \cup \overline{F}'\) is convex in \(A\).
    \item[\textup{(iii)}] There exists a \(k\)-dimensional affine subspace of \(A\) containing \(F\) and \(F'\).
\end{itemize}
\end{lemma}

\begin{proof}
Let \(H=\operatorname{aff}(F'')\).  In the normal quotient by the direction
space of \(H\), the two facets determine rays.  Opposition in the link means
that these rays are opposite, which is equivalent to the existence of a
\(k\)-dimensional affine subspace containing both facets.  This proves
\textup{(i)}\(\Leftrightarrow\)\textup{(iii)}, and the same
transverse-segment argument as in Lemma~\ref{opp-convex-sph} gives
\textup{(ii)}\(\Rightarrow\)\textup{(iii)}.

For the converse, let \(v_0,\ldots,v_k\) be the vertices of \(F\), with
\(F''\) spanned by \(v_1,\ldots,v_k\), and let \(p\) be the orthogonal
projection of \(v_0\) to \(H\).  For each \(i=1,\ldots,k\), the stabilizer
\(W_{v_i}\) of \(v_i\) in the affine Weyl group is a finite reflection group
acting on the link of \(v_i\).  After translating \(v_i\) to the origin, the
vector \(p-v_i\) is the orthogonal projection of \(v_0-v_i\) onto
\(H-v_i\).  The spherical projection argument of
Lemma~\ref{opp-convex-sph}, applied in this link, therefore shows that \(p\)
lies in the convex cone \(C_i\) at \(v_i\) generated by the rays from \(v_i\)
toward the other vertices of \(F''\).  Since
\[
\bigcap_{i=1}^k C_i=\operatorname{conv}(v_1,\ldots,v_k),
\]
we have \(p\in\overline{F''}\).  The displayed equality is elementary: in
barycentric coordinates \(t_1,\ldots,t_k\) on \(H\), the cone \(C_i\) is
characterized by \(t_j\geq0\) for \(j\neq i\), and their intersection is
therefore characterized by \(t_j\geq0\) for all \(j\).  The same holds for the
projection of the extra vertex of \(F'\).  Under condition (i) the two extra
vertices lie on opposite sides of \(H\), so the final segment-decomposition
argument in the proof of Lemma~\ref{opp-convex-sph}, with cones replaced by
simplices, proves that \(\overline F\cup\overline{F'}\) is convex.
\end{proof}

\subsection{A Standard Model for Geometric Geodesics in an Apartment}\label{s:apt}

We work in an explicit affine apartment model in order to construct the standard
geometric geodesic and verify its convexity directly.

We model the apartment \(A\) as the quotient vector space
\(\mathbb{R}^n / \Delta \mathbb{R}\), where
\(\Delta : \mathbb{R} \to \mathbb{R}^n\) denotes the diagonal embedding
\(\Delta(t) = (t,\ldots,t)\).  Viewed in this way, \(A\) is an affine space.
For any distinct indices \(i,j\) with \(1 \le i \ne j \le n\), the function
\[
(x_1,\ldots,x_n)+\Delta\mathbb{R} \longmapsto x_i - x_j
\]
is well-defined on \(A\); we denote it by \(x_i - x_j\).

Within this framework, the affine root system is
\[
\{\, x_i - x_j + c \mid 1 \le i \ne j \le n,\ c \in \mathbb{Z} \,\}.
\]
A choice of simple affine roots is given by
\[
x_n - x_1 + 1,\; x_1 - x_2,\; \ldots,\; x_{n-1} - x_n.
\]
Let \(v_0,\ldots,v_{n-1}\) be the corresponding vertices, where
\[
v_i=(\underbrace{1,\ldots,1}_{i},0,\ldots,0)+\Delta\mathbb{R}.
\]

Fix an ordered partition \((\lambda_0,\ldots,\lambda_k)\) of \(n\), and set
\(\Lambda_i = \sum_{j=0}^{i-1}\lambda_j\) with \(\Lambda_0=0\).
We define a \(k\)-facet \(C\) with vertices
\[
w_i=v_{\Lambda_i}, \qquad i=0,\ldots,k.
\]
This facet lies in a \(k\)-dimensional affine subspace \(A'\subset A\),
determined by the equalities \(x_i-x_{i+1}=0\) for all
\(i\notin\{\Lambda_1,\ldots,\Lambda_k\}\).

To parametrize points in \(A'\), we introduce coordinates
\[
y_i=x_{\Lambda_i}-x_n, \qquad i=1,\ldots,k.
\]
With this notation, the vertex \(w_i\) has coordinates
\[
(y_1,\ldots,y_k)(w_i)
=
(\underbrace{1,\ldots,1}_{i},\underbrace{0,\ldots,0}_{k-i}).
\]

We now extend this construction to an infinite sequence of vertices.
For each \(i\in\mathbb{Z}\), define \(w_i\in A'\) by
\[
(y_1,\ldots,y_k)(w_i)
=
\left(
\left\lceil \tfrac{i}{k} \right\rceil,
\ldots,
\left\lceil \tfrac{i-(k-1)}{k} \right\rceil
\right).
\]
Equivalently, \(w_{i+k}=w_i+w_k\) for all \(i\).  For each \(i\in\mathbb{Z}\),
let \(C_i\) denote the \(k\)-facet with vertices \(w_i,\ldots,w_{i+k}\).
The sequence \((C_i)_{i\in\mathbb{Z}}\) should be viewed as a candidate
\emph{standard geometric geodesic} in the apartment.

To describe these facets uniformly, we extend the functions \(y_i\) to all
\(i\in\mathbb{Z}\) by imposing the rule \(y_{i+k}=y_i-1\).
Then \(\overline{C}_i\) is characterized by the inequalities
\[
1 \ge y_{i+1} \ge \cdots \ge y_{i+k} \ge 0.
\]
In particular, \(\overline{C}_i\) lies in the region
\[
\grC = \{\, y_j \ge y_{j+1} \mid j\in\mathbb{Z} \,\} \subset A'.
\]

We further define two auxiliary convex sets inside \(\grC\):
\[
\overline{C}_i^+ = \{\, y_{i+k} \ge 0 \,\}, \qquad
\overline{C}_i^- = \{\, 1 \ge y_{i+1} \,\}.
\]

The next lemma describes how the facets \(C_i\) decompose these regions and how
they fit together along the geodesic.

\begin{lemma}
For any \(i\in\mathbb{Z}\),
\[
C_i \subset
\overline{C}_i^+ \setminus \overline{C}_{i+1}^+
\subset \overline{C}_i,
\qquad
C_i \subset
\overline{C}_i^- \setminus \overline{C}_{i-1}^-
\subset \overline{C}_i.
\]
Moreover,
\[
\overline{C}_i^+ = \bigcup_{j\ge i}\overline{C}_j,\qquad
\overline{C}_i^- = \bigcup_{j\le i}\overline{C}_j,\qquad
\grC = \bigcup_{j\in\mathbb{Z}}\overline{C}_j.
\]
\end{lemma}

\begin{proof}
First we prove the local difference statement.
The defining inequalities of \(C_i\) include those defining
\(\overline{C}_i^+\), so \(C_i\subset \overline{C}_i^+\).  Since
\(1>y_{i+1}\) is a defining condition of \(C_i\) and
\(y_{i+1}-1=y_{i+k+1}\geq 0\) is a defining condition of
\(\overline{C}_{i+1}^+\), \(C_i\) is disjoint from
\(\overline{C}_{i+1}^+\).  We have proved: \(C_i
\subset\overline{C}_i^+\setminus \overline{C}_{i+1}^+\).

Consider a point in \(\overline{C}_i^+\setminus
\overline{C}_{i+1}^+\).  It has to satisfy \(y_{i+k} \geq 0\) and
\(y_{i+k+1} < 0\).  The latter inequality is equivalent to \(1>
y_{i+1}\).  Therefore, the point lies in \(\overline{C}_i\).  We have
proved: \(\overline{C}_i^+\setminus
\overline{C}_{i+1}^+\subset\overline{C}_i\).  By induction we have 
\[
\overline{C}_i^+\setminus \overline{C}_{i+s}^+\subset \overline{C}_i\cup\cdots\cup\overline{C}_{i+s-1}
\]
for any \(s \geq 1\).

We now pass from the local difference statement to the infinite union.
Consider an arbitrary point \(p\) of \(\overline{C}_i^+\).  Then
\(y_0(p) <N\) for any integer \(N\) large enough.  Therefore,
\(y_{Nk}(p) = y_0(p)-N<0\) so \(p \notin \overline{C}_j^+\) for \(j\) large
enough, \(j > i\).  The preceding paragraph implies \(p \in \bigcup_{j
  \geq i} \overline{C}_j\).  This shows \(\overline{C}_i^+ \subset
\bigcup_{j\geq i}\overline{C}_j\).  This inclusion is an equality
because each \(\overline{C}_j\) with \(j\ge i\) is contained in
\(\overline{C}_i^+\).

The statements regarding \(\overline{C}_i^-\) are proved similarly.
The equality \(\grC=\bigcup_{j\in \bbZ}\overline{C}_j\) follows from
the argument of the last paragraph.
\end{proof}

\begin{remark}[The standard geodesic is a straight strip]
Let \(e=(1,\ldots,1)\) in the \(y\)-coordinates on \(A'\), and put
\[
d_i=y_i-y_{i+1}\quad(1\le i<k),
\qquad
d_k=y_k-y_1+1.
\]
The defining inequalities of \(\grC\) are equivalent to
\(d_i\ge0\) for all \(i\), while
\(
d_1+\cdots+d_k=1
\).
They are invariant under translation in the direction \(e\).  If
\(e^\perp\) denotes the orthogonal complement with respect to the Euclidean
metric on the apartment, then
\[
S:=\grC\cap e^\perp
\]
is a Euclidean \((k-1)\)-simplex, and orthogonal projection gives an
isometric splitting
\[
\grC\simeq S\times\mathbb R.
\]
Together with the preceding lemma, this says that the bi-infinite standard
geometric \(k\)-geodesic is a straight simplicial strip paved by the facets
\(C_i\).  The finite geodesics below are its consecutive convex segments.
\end{remark}

We now show that this combinatorial chain of facets indeed behaves as a genuine
geometric geodesic, in the sense that convexity agrees with taking consecutive
unions.

\begin{theorem}\label{thm:standard-geodesic}
For any integers \(i \le j\), the convex hull of
\(\overline{C}_i \cup \overline{C}_j\) is
\[
\overline{C}_i^+ \cap \overline{C}_j^-
=
\bigcup_{i\le s\le j}\overline{C}_s.
\]
In particular, the sequence of facets
\(
(C_i, C_{i+1}, \ldots, C_j)
\)
forms a geometric geodesic.
\end{theorem}

\begin{proof}
The equality
\(\overline{C}^+_i\cap \overline{C}^-_j =
  \bigcup_{i\leq s\leq j}\overline{C}_s\)
follows by intersecting the two union descriptions in the preceding lemma.
Both \(\overline{C}^+_i\) and \(\overline{C}^-_j\) are convex, hence their
intersection is a convex set containing \(\overline{C}_i\) and
\(\overline{C}_j\), so it
contains the convex hull \(H\) of \(\overline{C}_i\cup
\overline{C}_j\).  To prove
the other inclusion, we do an induction on \(j-i\).  The result is
clear when \(j-i=0\).  When \(j-i=1\), the preceding lemma gives
\[
\overline{C}_i\cup\overline{C}_{i+1}
=\overline{C}_i^+\cap \overline{C}_{i+1}^-,
\]
which is convex; hence it is exactly the convex hull of
\(\overline{C}_i\cup\overline{C}_{i+1}\).

To prove minimality of this convex set, we show that every convex set
containing the endpoints also contains each intermediate facet.
Since \(H\) contains \(\overline{C}_j\), it contains \(w_j,\ldots,w_{j+k}\).  We claim that it also
contains \(w_{j-1}\).  Indeed, let \(l\) be the unique integer such that
\(l\equiv j-1\pmod{k}\) and \(i\leq l<i+k\).  Then \(w_l\in
\overline{C}_i\).  Write \(j-1=l+rk\) with \(r\ge0\).  Since
\(w_{a+k}=w_a+w_k\), we have
\[
w_{j-1}=w_l+r w_k,\qquad
w_{j+k-1}=w_l+(r+1)w_k.
\]
Thus \(w_{j-1}\) lies on the line segment joining \(w_l\) and
\(w_{j+k-1}\), with both endpoints lying on \(H\).
Therefore, \(w_{j-1} \in H\).  Since \(H\) now contains all the vertices
from \(w_{j-1}\) through \(w_{j+k-1}\), it contains
\(\overline{C}_{j-1}\).  By the induction hypothesis it then contains
\[
\bigcup_{i\leq s\leq j-1}\overline{C}_s,
\]
and hence also \(\bigcup_{i\leq s\leq j}\overline{C}_s\).  This completes the
proof.
\end{proof}

\subsection{Normal Forms of Geometric Geodesics} \label{s:normal-forms}

We now analyze how a given \(k\)-facet in an apartment can be extended to larger
convex configurations.  The goal of this subsection is to understand which local
extensions are compatible with convexity, and to show that these constraints
force any geometric geodesic to follow a unique global pattern.

Fix the standard \(k\)-facet \(C=C_0\subset A\) introduced in
Section~\ref{s:apt}.  Consider the \(k\)-facets \(C'\subset A\) adjacent to
\(C\), i.e.\ sharing a \((k-1)\)-subfacet with \(C\), such that
\(\overline{C}\cup\overline{C'}\) is convex.  By
Lemma~\ref{opp-convex-aff}, there are exactly \(k+1\) such facets.  Each arises by
choosing a \((k-1)\)-facet \(D\) of \(C\) and taking the facet in the link of \(D\)
that is opposite to \(C\).

For \(i=0,\ldots,k\), let \(C^{(i)}\) denote the facet obtained by using the
\((k-1)\)-facet of \(C\) whose closure does not contain the vertex \(w_i\).
Explicitly, \(C^{(i)}\) has vertices
\[
\{w_i'\}\cup\{w_0,\ldots,w_k\}\setminus\{w_i\},
\]
where \(w_i'\in A'\) is given in the \(y\)-coordinates by
\[
(y_1,\ldots,y_k)(w_i')=
\begin{cases}
(2,1,\ldots,1), & i=0,\\[2pt]
(\underbrace{1,\ldots,1}_{i-1},0,1,0,\ldots,0), & 1\le i\le k-1,\\[2pt]
(0,\ldots,0,-1), & i=k.
\end{cases}
\]
Note that \(C^{(0)} =C_1\) and \(C^{(k)}=C_{-1}\).

We first classify the possible convex extensions of a fixed adjacent pair of
\(k\)-facets.  Since any geometric geodesic of length \(2\) can, after
rechoosing the standard apartment and applying an affine Weyl group symmetry,
be identified with the adjacent pair \((C_{-1},C_0)\), it suffices to analyze
extensions of this standard configuration.

\begin{lemma}\label{lem:geo-3}
Let \(C'\) be a \(k\)-facet in \(A\).  Suppose the triple of \(k\)-facets
\[
(C_{-1},\,C_0,\,C')
\]
forms a geometric geodesic in the apartment \(A\).
Then the third facet \(C'\) must be one of the two facets
\[
C' = C^{(0)} (= C_1) \quad\text{or}\quad C' = C^{(k-1)}.
\]
\end{lemma}

\begin{remark}
The two alternatives in the lemma are related by a symmetry of the local
configuration.  The affine transformation
\[
\tau(y_1,\ldots,y_k)
=
(1-y_k,y_1-y_k,\ldots,y_{k-1}-y_k)
\]
fixes \(C\) setwise and satisfies
\[
\tau(w_i)=w_{i+1}\quad(0\leq i<k),
\qquad
\tau(w_k)=w_0,
\]
and
\[
\tau(C^{(i)})=C^{(i+1)}
\qquad(0\leq i\leq k),
\]
where the superscripts in the last formula are taken modulo \(k+1\).
Consequently, \(\tau\) sends
\[
(C^{(k)},C,C^{(k-1)})
\quad\text{to}\quad
(C^{(0)},C,C^{(k)}),
\]
which is the reverse of \((C^{(k)},C,C^{(0)})\).  Thus the two possibilities
admitted by the lemma are equivalent up to this affine symmetry and reversal.
\end{remark}

\begin{proof}
Since \(C_0\cap C'\) is a \((k-1)\)-facet and
\(\overline{C_0}\cup\overline{C'}\) is convex, Lemma~\ref{opp-convex-aff} implies
that \(C'=C^{(i)}\) for some \(i\in\{0,\ldots,k\}\).  As \(C_{-1}\) is the facet
opposite to \(C_0\) across the \((k-1)\)-subfacet of \(C_0\) missing \(w_k\), we
may (and do) identify \(C_{-1}=C^{(k)}\).  The assumption that \(C'\) is distinct
from \(C_{-1}\) excludes \(i=k\).  Hence \(0\le i\le k-1\).

\smallskip
\noindent\emph{Step 1: the standard convex triple.}
For \(i=0\), we have \(C^{(0)}=C_1\), hence
\[
\overline{C_{-1}}\cup\overline{C_0}\cup\overline{C^{(0)}}
=\overline{C}_{-1}\cup\overline{C}_0\cup\overline{C}_1,
\]
which is convex by Theorem~\ref{thm:standard-geodesic}.  Thus \(i=0\) is always
allowed in the standard model.

\smallskip
\noindent\emph{Step 2: excluding the intermediate cases \(1\le i\le k-2\).}
Fix \(i\) with \(1\le i\le k-2\), and set
\[
U_i' := \overline{C_{-1}}\;\cup\;\overline{C_0}\;\cup\;\overline{C^{(i)}}
= \overline{C^{(k)}}\;\cup\;\overline{C}\;\cup\;\overline{C^{(i)}}.
\]
Let \(z\) be the midpoint of \(w_i'\) and \(w_k'\).  Then \(z\) lies on the line
segment joining \(w_i'\in\overline{C^{(i)}}\) and
\(w_k'\in\overline{C^{(k)}}=\overline{C_{-1}}\), so \(z\) lies in the convex hull
of \(U_i'\).  We claim \(z\notin U_i'\), which will show that the candidate
union is not convex.

Indeed, \(y_k(z)=-\tfrac12\), so \(z\notin\overline{C_0}=\overline{C}\), since
\(\overline{C}\) is defined by \(1\ge y_1\ge\cdots\ge y_k\ge 0\).
The same inequality shows \(z\notin\overline{C^{(i)}}\), since
\(\overline{C^{(i)}}\subset\{y_k\ge 0\}\) (equivalently, it satisfies a chain of
inequalities ending with \(y_k\ge 0\)).
Finally, \(z\notin\overline{C^{(k)}}=\overline{C_{-1}}\), since at \(z\) we have
\(y_i(z)=0\) and \(y_{i+1}(z)=\tfrac12\), which violates the monotonicity
conditions defining \(\overline{C^{(k)}}\).
Thus \(z\notin U_i'\), proving non-convexity for \(1\le i\le k-2\).

The remaining case \(i=k-1\) is admissible by the symmetry in the preceding
remark, since its triple is carried to the reverse of the convex triple in
Step~1.
\end{proof}
We next record the corresponding rigidity statement for four \(k\)-facets.
By the lemma and the preceding symmetry, any three consecutive facets of a
geometric geodesic can, after rechoosing the apartment and the standard
\(k\)-facet, be identified with one of the two oriented configurations
\[
(C_{0},\,C_{1},\,C_2)
\qquad\text{or}\qquad
(C_0,\,C_{-1},\,C_{-2}),
\]
where the displayed triple is a geometric geodesic of the type classified
in Lemma~\ref{lem:geo-3}.  These two configurations correspond to the two
possible orientations of a geometric geodesic.

In the following lemma, we fix one orientation and show that, once three
consecutive facets are fixed, the geodesic admits a unique continuation in that
direction.

\begin{lemma}[Rigidity of geometric extension]\label{lem:geo-4}
Let \(C'\) be a \(k\)-facet in \(A\).  Suppose the sequence of \(k\)-facets
\[
(C_{0},\,C_{1},\,C_2,\,C')
\]
forms a geometric geodesic in the apartment \(A\).
Then necessarily
\[
C' = C_3.
\]
In particular, once three consecutive facets of a geometric geodesic are fixed,
its continuation in the chosen orientation is uniquely determined.
\end{lemma}

\begin{proof}
For \(k=1\), the assertion is the uniqueness of continuing a line segment in
the opposite direction across the common vertex and is immediate.  We
therefore assume \(k\ge2\).

Rechoosing indices if necessary, we may equivalently consider the configuration
\[
(C_{-2},\,C_{-1},\,C_0,\,C')
\]
and show that it forms a geometric geodesic if and only if \(C'=C_1\).
This reindexing puts the known three facets in the standard backward
orientation.

Applying Lemma~\ref{lem:geo-3} to the triple \((C_{-1},C_0,C')\), we obtain
\[
C' = C^{(0)}=C_1
\qquad\text{or}\qquad
C' = C^{(k-1)}.
\]

If \(C'=C_1\), then the sequence
\((C_{-2},\,C_{-1},\,C_0,\,C_1)\) forms a geometric geodesic by
Theorem~\ref{thm:standard-geodesic}.  It remains to exclude the case
\(C'=C^{(k-1)}\).

Assume \(C'=C^{(k-1)}\).
Let \(w_{-2}\) be the unique vertex of \(C_{-2}\) not contained in \(C_{-1}\),
with
\[
(y_1,\ldots,y_k)(w_{-2})=(0,\ldots,0,-1,-1).
\]
Let \(w'_{k-1}\) be the extra vertex of \(C^{(k-1)}\), and let \(z\) be the midpoint
of \(w_{-2}\) and \(w'_{k-1}\).  Then
\[
(y_1,\ldots,y_k)(z)=(\tfrac12,\ldots,\tfrac12,-\tfrac12,0).
\]
The point \(z\) lies in the convex hull of
\(\overline{C_{-2}}\cup\overline{C_{-1}}\cup\overline{C_0}\cup\overline{C^{(k-1)}}\),
but \(z\) does not lie in any of the four closures in this union.  Indeed,
\(z\notin\overline{C^{(k-1)}}\) since \(y_{k-1}(z)<0\).  For the standard
facets, recall that the extended coordinates satisfy
\(y_0=y_k+1\) and \(y_{-1}=y_{k-1}+1\).  Thus at \(z\) we have
\[
y_{-1}(z)=\tfrac12,\qquad y_0(z)=1,\qquad
y_{k-1}(z)=-\tfrac12,\qquad y_k(z)=0.
\]
Now \(z\notin\overline{C_0}\) because the defining inequalities for
\(\overline{C_0}\) include \(y_{k-1}\ge y_k\), which fails at \(z\).
Also \(z\notin\overline{C_{-1}}\), since the defining inequalities for
\(\overline{C_{-1}}\) include \(y_{k-1}\ge0\).  Finally,
\(z\notin\overline{C_{-2}}\), since the defining inequalities for
\(\overline{C_{-2}}\) include \(y_{-1}\ge y_0\), which also fails at \(z\).
Hence the union of the four closures is not convex, contrary to the defining
convex-hull condition of a geometric geodesic.

Therefore \(C'=C^{(k-1)}\) is impossible, and we conclude \(C'=C^{(0)}=C_1\).
\end{proof}

Note that if we choose the opposite orientation, the same argument shows that
the geometric triple
\[
(C_0,C_{-1},C_{-2})
\]
extends uniquely to \((C_0,C_{-1},C_{-2},C_{-3})\).
Combining the local convexity constraints for triples and quadruples therefore
yields a global normal form for geometric geodesics.

\begin{corollary}[Normal form for geometric geodesics in an apartment]
\label{cor:geo-normal-form}
Let \(m\ge0\), and let \((\tilde C_0, \tilde C_1, \ldots,\tilde C_m)\) be a
geometric geodesic of \(k\)-facets in \(\scrB\).  After choosing an apartment
\(A\) containing the geodesic and rechoosing the standard \(k\)-facet in \(A\),
one of the following holds:
\[
\begin{aligned}
\tilde C_i &= C_i &&\text{for all } i=0,1,\ldots,m,\\
\intertext{or}
\tilde C_i &= C_{-i} &&\text{for all } i=0,1,\ldots,m.
\end{aligned}
\]
\end{corollary}

\begin{proof}
Choose an apartment containing the endpoint facets \(\tilde C_0\) and
\(\tilde C_m\). Since apartments are convex and the convex hull of the endpoints
is, by definition of a geometric geodesic, the union of the intervening facet
closures, all facets \(\tilde C_i\) lie in this apartment.
We may therefore apply the local classification inside this apartment.
For $m<3$, the result follows from the preceding lemmas directly. For $m\geq 3$,
Lemma~\ref{lem:geo-4} then gives a unique continuation in the forward
(respectively, backward) direction, and induction yields the claimed normal
form.
\end{proof}

The normal form will be compared with successor chains in Section~5.

\section{Combinatorial \texorpdfstring{\(k\)-geodesics}{k-geodesics}}
\label{s:comb-geodesic}

We now bring the successor relation into the picture.  A combinatorial geodesic
is a successor chain of pointed \(k\)-facets.  Such chains are the paths that
will later be counted by adjacency operators, so the main question of this
section is how they compare with the geometric geodesics of Section~4.

\subsection{Definition of combinatorial \texorpdfstring{\(k\)-geodesics}{k-geodesics}}
Let \((\dot{C}_0, \ldots, \dot{C}_m)\) be a sequence of pointed \(k\)-facets. We say that \((\dot{C}_0, \ldots, \dot{C}_m)\) forms a \textit{combinatorial geodesic} if \(\dot{C}_{i+1}\) is a successor of \(\dot{C}_i\) for each \(i = 0, \ldots, m-1\).

We allow either kind of geodesic to be indexed by any nonempty interval
\[
\{i\in\mathbb Z:a<i<b\},
\qquad a,b\in\mathbb Z\cup\{-\infty,\infty\}.
\]

\medbreak

The following theorem compares these successor chains with the geometric notion
developed in the previous section.
For sequences consisting of one or two facets, the equivalence of the
underlying geometric and combinatorial conditions follows immediately from
Lemma~\ref{opp-convex-aff} and the definition of a successor.  The orientation
assertion below is stated for sequences of at least three facets.
\begin{theorem}\label{twogeodesics}
Let \(1\le k\le n-1\), and assume $m\ge 2$.
\begin{itemize}
   \item[\textup{(i)}] Let \((C_0, \ldots, C_m)\) be a geometric geodesic of \(k\)-facets.
   For each of the two orientations
   \[
      (C_0,\ldots,C_m)
      \quad\text{and}\quad
      (C_m,\ldots,C_0),
   \]
   there is at most one pointing of the facets in that order which can make the
   oriented sequence into a combinatorial geodesic.  Moreover, for \(k\ge2\),
   exactly one of the two orientations admits such a pointing.  For \(k=1\), both
   orientations admit such pointings, and both oriented pointed sequences are
   combinatorial geodesics.

   \item[\textup{(ii)}] Let \((\dot{C}_0, \ldots, \dot{C}_m)\) be a combinatorial geodesic of
   pointed \(k\)-facets, and let \(C_i\) denote the underlying \(k\)-facet of \(\dot{C}_i\).
   Then \((C_0, \ldots, C_m)\) is a geometric geodesic of \(k\)-facets.
\end{itemize}
\end{theorem}

\noindent
The same assertions hold for geodesics indexed by any interval
\(\{\,i\in\mathbb{Z}: a<i<b\,\}\) containing at least three facets; in
particular they apply to one-sided and bi-infinite geodesics.

\begin{proof}
We prove part \textup{(i)}.  The proof of part \textup{(ii)} will be given in the next
subsection.

Let \((C_0,\ldots,C_m)\) be a geometric geodesic of \(k\)-facets.
If \(k=1\), then each facet has two vertices.  For a geometric geodesic of
edges, point \(C_i\), for \(i<m\), at the vertex not contained in \(C_{i+1}\)
in the chosen orientation; the pointing of the last edge is then forced by the
successor relation from \(C_{m-1}\) to \(C_m\).  Consecutive edges have convex
union, so the corresponding vertices in the relevant link are opposite by
Lemma~\ref{opp-convex-aff}; hence the successor condition holds.  The same
argument applies after reversing the orientation, proving the claim in this
case.

\medskip
Assume from now on that \(k\ge2\).
Since the statement is symmetric under reversing the order of the facets,
we may replace the sequence by its reverse if necessary.
By Corollary~\ref{cor:geo-normal-form}, we may therefore assume that
\((C_0,\ldots,C_m)\) is the standard geometric geodesic, where
\(C_i\) is the \(k\)-facet with vertices \(w_i,\ldots,w_{i+k}\)
as in Section~\ref{s:apt}.

For each \(i\), define the forward pointed facet
\[
   \dot{C}_i^+ := (C_i, w_i).
\]
Equivalently, \(\dot{C}_i^+\) is represented by the ordered vertex sequence
\[
   (w_i, w_{i+1}, \ldots, w_{i+k}).
\]

By construction, \(w_i\) is the unique vertex of \(C_i\) not contained in
\(C_{i+1}\), and \(w_{i+1}\) is the distinguished vertex of \(C_{i+1}\).
The union \(\overline C_i\cup\overline C_{i+1}\) is convex by
Theorem~\ref{thm:standard-geodesic}, so the two facets are opposite across
their common face by Lemma~\ref{opp-convex-aff}.  Hence
\(\dot{C}_{i+1}^+\) is the successor of \(\dot{C}_i^+\) for all \(i\),
and \((\dot{C}_0^+,\ldots,\dot{C}_m^+)\) is a combinatorial geodesic.

Uniqueness follows from the fact that, given the underlying facets \(C_i\),
the overlap condition forces the basepoint of each forward successor uniquely.

For the reverse orientation, the only possible pointing is forced in the same
way.  Namely, for each \(i=0,\ldots,m\) define
\[
   \dot{C}_i^-=(C_i,w_{i+k}),
\]
since \(w_{i+k}\) is the unique vertex of \(C_i\) not contained in \(C_{i-1}\)
for \(i\ge1\).
The successor relation in particular requires that
\(w_{m+k-1}\) be the successor of \(w_{m+k}\) in the cyclic ordering of the
vertices of \(C_m\).
But the successor of \(w_{m+k}\) on \(C_m\) is \(w_m\).  Hence this requirement
would force \(w_{m+k-1}=w_m\), which can occur only when \(k=1\).

In that case,
\[
   (\dot C_m^-,\dot C_{m-1}^-,\ldots,\dot C_0^-)
   =
   \bigl((C_m,w_{m+1}),\,(C_{m-1},w_m),\,\ldots,\,(C_0,w_1)\bigr)
\]
is indeed a combinatorial geodesic.

For \(k\ge 2\), the cyclic ordering of vertices of a \(k\)-facet prevents this
successor relation from holding, and hence exactly one orientation of a
geometric geodesic admits a compatible pointing.

\end{proof}

\subsection{Stabilizers and transitivity}

The results of the previous sections show that geometric geodesics of
\(k\)-facets admit a rigid normal form and that, except when \(k=1\),
a geometric geodesic determines a unique orientation compatible with the
successor relation.  In particular, there exists a distinguished
bi-infinite sequence of pointed \(k\)-facets
\[
   \{\dot C_i\}_{i\in\mathbb Z},
\]
called the \emph{standard combinatorial geodesic}, whose underlying sequence of
\(k\)-facets \(\{C_i\}_{i\in\mathbb Z}\) is the standard geometric geodesic
introduced in Section~\ref{s:apt}.

The main result of this section is the following:
\emph{every combinatorial geodesic in the building with a fixed type at one
chosen index is obtained
from the standard one of that type by the action of \(\GL_F(V)\).}
\medskip

More precisely, we will show that the standard combinatorial geodesic serves as
a universal model: its stabilizer inside \(\GL_F(V)\) can be described
explicitly, and the ambient group \(\GL_F(V)\) acts transitively on the set of
combinatorial geodesics whose type is fixed at one chosen index.  As a
consequence, the underlying sequence of facets of any combinatorial geodesic is
automatically a geometric geodesic, completing the proof of
Theorem~\ref{twogeodesics}\textup{(ii)}.

\medskip

To make this precise, we now pass to a concrete lattice model.
Fix a basis \(u_1,\ldots,u_n\) of \(V:=F^n\), so that
\(\GL_F(V)=\GL_n(F)=G_\gl\).  Let
\(T\) be the maximal torus of \(\GL_F(V)\) consisting of those
elements diagonalized by this basis.  For \(x \in \bbZ^n\), let
\(L(x)\) be the \(\ringO\)-lattice spanned by
\(\pi^{x_1}u_1,\ldots,\pi^{x_n}u_n\).  Then \(x+\Delta\bbZ\mapsto [L(x)]\) is a bijection between \(\bbZ^n/ \Delta \bbZ\) and vertices on
\(A(T)\), the apartment associated to \(T\).  The set \(\bbZ^n/\Delta \bbZ\)
is simply the set of vertices on \(A=\bbR^n/\Delta \bbR\), the apartment
considered in the last section.  The above bijection in fact extends
to an isomorphism \(\scrL :A\simto A(T)\).

Recall that we have defined \(\{w_i\}_{i\in\bbZ}\) and \(\{\dot
C_i\}_{i\in\bbZ}\) in Section~\ref{s:apt}.
We would like to compute the stabilizer of the combinatorial geodesic
\(\{\scrL(\dot C_i)\}_{a'<i<b'}\), where \(a',b'\in
\bbZ\cup\{-\infty,\infty\}\).  Equivalently, this is to compute the
stabilizer of \(\{\scrL(w_i)\}_{a<i<b}\), where \(a,b\in
\bbZ\cup\{-\infty,\infty\}\).  We emphasize that the geodesics are
indexed sequences, and to fix a sequence is to fix each member of the sequence.

Let \(\tilde w_i\in \bbZ^n\) be such
that \(\tilde w_i+\Delta\bbR=w_i\) and the \(n\)-th coordinate of
\(\tilde w_i\) is zero.  We will in fact deal with the stabilizer of
\(\{L(\tilde w_i)\}_{a<i<b}\).  This stabilizer, together with
\(Z(\GL_F(V))\), generate the stabilizer of
\(\{\scrL(w_i)\}_{a<i<b}\).

Let \(M_j=\ringO\langle u_{\lambda_0+\cdots +\lambda_{j-2}+s}:1\leq
s\leq \lambda_{j-1}
  \rangle\), \(j=1,\ldots,k+1\).  Then
\[
  L(\tilde
  w_i)=\pi^{y_1(w_i)}M_1\oplus \cdots \oplus \pi^{y_k(w_i)}
  M_{k}\oplus M_{k+1}.
\]
Here $y_i$ are the coordinate functions introduced in Section~\ref{s:apt}.

Put \(V_j=M_j\otimes _{\ringO} F\) so that \(V=V_1\oplus\cdots\oplus
V_{k+1}\).  We regard an element \(a\) of \(\Hom(V,V)\) as
\((k+1)\times(k+1)\) matrix with \(a_{ij} \in \Hom(V_j,V_i)\).  The following
description is obtained by imposing the conditions
\(aL(\tilde w_i)=L(\tilde w_i)\) for all \(i\ge0\), respectively for all
\(i\le0\), and reading off the resulting valuation bounds on the blocks
\(a_{ij}\).

\begin{lemma}\label{lem:half-stabilizers}
With respect to the decomposition \(V=V_1\oplus\cdots\oplus V_{k+1}\), write
\(a=(a_{ij})\).  The stabilizer of \(\{L(\tilde w_i)\}_{i\geq 0}\) in
  \(\GL_F(V)\) consists precisely of those \(a\in\GL_F(V)\) such that 
\[
  a_{ij} \in \begin{cases}
    \GL_\ringO(M_i), &i=j,\\
  \Hom_\ringO(M_j,M_i),&i>j,\\
  \pi\Hom_\ringO(M_j,M_i),&i<j\neq k+1,\\
  \{0\}, &i<j=k+1.
\end{cases}
\]  The stabilizer of \(\{L(\tilde w_i)\}_{i\leq 0}\) in \(\GL_F(V)\) consists
precisely of those \(a\in\GL_F(V)\) such that 
\[
  a_{ij} \in \begin{cases}
    \GL_\ringO(M_i), &i=j,\\
    \Hom_\ringO(M_j,M_i),&k+1\neq i>j,\\
     \{0\},&k+1 = i>j,\\
  \pi\Hom_\ringO(M_j,M_i),&i<j\neq k+1,\\
 \Hom_\ringO(M_j,M_i), &i<j=k+1.
\end{cases}
\]
\end{lemma}

\begin{proof}
Put \(y_{k+1}(w_m)=0\).  A block
\(a_{rs}\in\Hom_F(V_s,V_r)\) preserves all the lattices in the indicated
half-geodesic precisely when its valuation satisfies
\[
\operatorname{val}(a_{rs})\ge
\sup_m\bigl(y_r(w_m)-y_s(w_m)\bigr),
\]
where \(m\ge0\) or \(m\le0\), respectively.  For \(m\ge0\), these suprema
are \(0\) when \(r\ge s\), \(1\) when \(r<s\le k\), and infinite when
\(r<s=k+1\).  For \(m\le0\), they are \(0\) when
\(k+1\ne r>s\), infinite when \(r=k+1>s\), \(1\) when
\(r<s\ne k+1\), and \(0\) when \(r<s=k+1\).  The diagonal blocks must be
invertible over \(\ringO\).  These bounds give exactly the two displayed
block descriptions.
\end{proof}

To illustrate, we take \(k=3\) and represent the two stabilizers in the lemma in the
following form: 
\[
  \begin{bmatrix}
    1^\times &\pi &\pi &0\\
    1&1^\times&\pi&0\\
    1&1&1^\times&0\\
    1&1&1&1^\times
  \end{bmatrix}\quad\mbox{ and }\quad
   \begin{bmatrix}
    1^\times &\pi &\pi &1\\
    1&1^\times&\pi&1\\
    1&1&1^\times&1\\
    0&0&0&1^\times
  \end{bmatrix},
\]
where the symbols \(1^\times,1,\pi\) and \(0\) at the \((i,j)\)-entry mean that
the \((i,j)\)-entry lies in \(\GL_\ringO(M_i)\),
\(\Hom_\ringO(M_j,M_i)\), \(\pi\Hom_\ringO(M_j,M_i)\), and \(\{0\}\),
respectively.

\begin{lemma}\label{lem:add-predecessor}
Let \((\dot C_i)_{i\in\mathbb Z}\) be the standard combinatorial geodesic.
Consider combinatorial geodesics \((\dot D_i)_{i\ge -1}\) such that
\(\dot D_i=\scrL(\dot C_i)\) for all \(i\ge0\).  The stabilizer of
\(\{L(\tilde w_i)\}_{i\geq 0}\) in \(\GL_F(V)\) acts transitively on this set.
\end{lemma}

\begin{proof}
Only the predecessor \(\dot D_{-1}\) is variable; the fixed tail determines all
other terms.
It must be of the form \((D_{-1},v)\), where \(D_{-1}\) has
vertices \(v,\scrL(w_0),\ldots,\scrL(w_{k-1})\).  So the whole
geodesic is determined by \(v\), a lattice class.  A lattice class
can serve as \(v\) if and only if it has a unique representative \(L\) such
that
\(L(\tilde w_0)\subset L \subset \pi^{-1}L(\tilde w_{k-1})\) and
\[
  \frac{\pi^{-1}L(\tilde w_{k-1})}{L(\tilde w_0) }=
  \frac{\pi^{-1}L(\tilde w_k)}{L(\tilde w_0)} \oplus \frac{L}{L(\tilde
    w_0)}.
\]
It remains to show that the lattices \(L\) satisfying this condition form a
single \(H\)-orbit, where
\(H\) is the stabilizer of \(\{L(\tilde w_i)\}_{i\geq 0}\) in
\(\GL_F(V)\).  

 We have \(\pi^{-1}L(\tilde w_{k-1})/L(\tilde w_0)\simeq \bar M_k
  \oplus \bar M_{k+1}\), where \(\bar M_j=\pi^{-1}M_j/M_j\) for
  \(j=k,k+1\).  Under this identification, \(\pi^{-1}L(\tilde
  w_k)/L(\tilde w_0)\) is
  \(\bar M_{k+1}\).  The set of \(L\)'s can then be identified with 
\[
\{U \subset \bar M_k\oplus \bar M_{k+1} : U\mbox{ is a subspace
  complement to } \bar M_{k+1}\}.
\]
The image of \(H\) in \(\GL(\bar M_k\oplus \bar M_{k+1})\) is exactly the
stabilizer of \(\bar M_{k+1}\) in \(\GL(\bar M_k\oplus \bar M_{k+1})\) by
the preceding lemma.  Indeed, the allowed matrix blocks in the preceding lemma
reduce to this parabolic stabilizer, and conversely any element of the parabolic
stabilizer lifts by choosing arbitrary \(\ringO\)-linear lifts in the allowed
blocks.  This group acts transitively on the above set of complements \(U\).
This finishes the proof.
\end{proof}

\begin{lemma}\label{lem:add-successor}
Let \((\dot C_i)_{i\in\mathbb Z}\) be the standard combinatorial geodesic.
Consider combinatorial geodesics \((\dot D_i)_{i\le1}\) such that
\(\dot D_i=\scrL(\dot C_i)\) for all \(i\le0\).  The stabilizer of
\(\{L(\tilde w_i)\}_{i\le0}\) in \(\GL_F(V)\) acts transitively on this set.
\end{lemma}

\begin{proof}
Only the successor \(\dot D_1\) is variable.  As in the proof of
Lemma~\ref{lem:add-predecessor}, its new terminal vertex is parametrized by
the complements of a fixed summand in a two-block quotient.  The second
stabilizer in Lemma~\ref{lem:half-stabilizers} reduces onto the parabolic
subgroup stabilizing that summand and therefore acts transitively on the set
of complements.  This proves the required transitivity without reversing the
successor relation.
\end{proof}

\begin{proposition}
Fix \(a,b,c\in\bbZ\cup\{-\infty,\infty\}\) with \(a<c<b\), and fix an
ordered partition \((\lambda_0,\ldots,\lambda_k)\) of \(n\).  The group
\(\GL_F(V)\) acts transitively on the set of combinatorial geodesics
\((\dot D_i)_{a<i<b}\) such that \(\dot D_c\) has type
\((\lambda_0,\ldots,\lambda_k)\).
\end{proposition}

\begin{proof}
When \(b-a=2\), the statement reduces to the transitivity of \(\GL_F(V)\) on
pointed facets of fixed type.  For finite \(a,b\), the general case follows by induction on
  \(b-a\).  After translating indices, Lemma~\ref{lem:add-predecessor} is the
  induction step for adding one predecessor to a fixed half-geodesic, while
  Lemma~\ref{lem:add-successor} is the corresponding step for adding one
  successor.  Iterating these two extensions gives transitivity on every
  finite window with the type fixed at the chosen index \(c\).

  Suppose now that one of the endpoints is infinite, and let
  \((\dot D_i)\) and \((\dot E_i)\) be two such geodesics.  Choose
  \(h\in\PGL_F(V)\) with \(h\dot D_c=\dot E_c\), and let \(K_c\) be
  the stabilizer of \(\dot E_c\) in \(\PGL_F(V)\); this is compact open.
  For an increasing exhaustion by finite windows \(I\) containing \(c\), set
  \[
  E_I=\{g\in K_c h:g\dot D_i=\dot E_i\text{ for every }i\in I\}.
  \]
  By the finite case, every \(E_I\) is nonempty.  These sets are closed and
  nested inside the compact set \(K_c h\), so their intersection is nonempty.
  An element of the intersection carries \((\dot D_i)\) to
  \((\dot E_i)\) on the whole interval.  Lifting it to \(\GL_F(V)\) proves
  the assertion.
\end{proof}

\subsection{Proof of Theorem~\ref{twogeodesics}(ii)}

We now prove part~(ii).

The underlying sequence of facets of the standard combinatorial geodesic
\(\{\scrL(\dot C_i)\}_{i\in\mathbb Z}\) is, by construction, a geometric
geodesic.
Since geometric geodesics are preserved under the action of
\(\GL_F(V)\), the same holds for the image of the standard geodesic under any
group element.
Now let \(\{\dot D_i\}_{a<i<b}\) be any combinatorial geodesic, choose an index
\(c\) in its indexing interval, and let \((\lambda_0,\ldots,\lambda_k)\) be the
type of \(\dot D_c\).  By the transitivity result above, \(\{\dot D_i\}_{a<i<b}\)
is the image of the standard combinatorial geodesic of this type.  Hence its
underlying sequence of facets is geometric.
This completes the proof.

Together, Sections~4 and 5 identify the geometric and combinatorial
descriptions as two realizations of the same \(k\)-geodesics in the building.
We henceforth use the term \emph{\(k\)-geodesic} without an adjective.  On a
quotient complex, we take the combinatorial description---a successor chain of
pointed facets---as the definition.

\section{Geodesics and zeta functions on quotient complexes}

We first define \(k\)-geodesics on quotient complexes without any finiteness
assumption.  Finiteness will be imposed only when the zeta functions are
introduced.

In this section, let $G=G_\ad$ for short.

\subsection{Quotient complexes}\label{quot-cx}
Throughout this section, we use the term \(\Delta\)-complex for what
Ranicki--Weiss call a \(\Delta\)-set.  Thus a \(\Delta\)-complex \(X\) consists
of sets \(X(k)\) of \(k\)-facets and face maps
\[
\partial_i:X(k)\longrightarrow X(k-1),\qquad 0\le i\le k,
\]
satisfying \(\partial_i\partial_j=\partial_{j-1}\partial_i\) for \(i<j\).  Its
geometric realization is a \(\Delta\)-complex in the sense of
Hatcher~\cite{Hat}*{2.1}.  We prefer Hatcher's term, but otherwise follow the
terminology and conventions of Section~2 of~\cite{RW}, with ``facet'' in place
of ``simplex.''  In particular, we regard \(X\) as its category of facets: a
morphism \(C\to D\) records a specified occurrence of \(C\) as a face of
\(D\).  Thus distinct face incidences are retained, even when they have the
same source and target.

If \(\Gamma\) acts on a \(\Delta\)-complex \(X\) by \(\Delta\)-complex
automorphisms, then the quotient \(\Gamma\backslash X\) inherits a natural
\(\Delta\)-complex structure; see \cite{RW}*{2.8}.  By contrast, the quotient
of a simplicial complex by a simplicial action need not inherit a
simplicial-complex structure.  For background on covering complexes, see
\cite{R}.

If \(X\) is a simplicial complex, then, by choosing a total order on the set of
vertices of each simplex of \(X\) in a compatible way, one can endow \(X\)
with the structure of a \(\Delta\)-complex; see \cite{RW}*{2.6}.  A simplicial
automorphism acts as an automorphism of this \(\Delta\)-complex precisely when
its induced bijection on the vertices of every simplex is order preserving.

\subsection{Quotient complexes \texorpdfstring{$\X$}{X}}
Recall that the set of vertex types of \(\scrB\) is \(\bbZ/n\bbZ\).  We now
give it the total order
\[
0<1<\cdots<n-1.
\]
This order is somewhat artificial, but a choice of total order is needed to
specify a \(\Delta\)-complex structure.  The vertices of every facet of
\(\scrB\) have distinct types, so ordering them by type is compatible with
passage to faces.  By Section~\ref{quot-cx}, the chosen order makes
\(\scrB\) into a \(\Delta\)-complex, which we denote by \(\scrB_\Delta\).
Every type-preserving automorphism of \(\scrB\) induces an automorphism of
\(\scrB_\Delta\).

Now let \(\Gamma\) be a discrete type-preserving subgroup of \(G\) acting
freely on \(\scrB\).  Then
\[
\scrB_\Gamma:=\Gamma\backslash\scrB_\Delta
\]
is a \(\Delta\)-complex.  In particular, \(\scrB_{\{1\}}=\scrB_\Delta\).  Since
\(\scrB\) is contractible, \(\Gamma\) is naturally
identified with the fundamental group of \(\scrB_\Gamma\).

Here are two simple facts that are useful for verifying the freeness condition
and, later, the finiteness condition.

\begin{itemize}
  \item If \(\Gamma\) is torsion-free, then \(\Gamma\) acts freely on
    \(\scrB\).  Indeed, for any facet \(F\) of \(\scrB\),
    \(\Gamma\cap\Stab_G(F)\), being discrete and compact, is finite,
    hence is equal to \(\{1\}\).
  \item If \(\Gamma\) is cocompact, then \(\Gamma\backslash G/K\) is
    finite for any compact open subgroup \(K\) of \(G\).  Thus if
    \(\Gamma\) is also type-preserving and acts freely, then
    \(\scrB_\Gamma\) is a finite \(\Delta\)-complex.
\end{itemize}

In the above, we have used the following facts: the set \(\scrB^{(k)}\)
of \(k\)-facets of \(\scrB_\Delta\) consists of finitely many
\(G\)-orbits, and the stabilizers are compact open subgroups.  The set
of \(k\)-facets of \(\scrB_\Gamma\) is
\(\Gamma\backslash\scrB^{(k)}\).

We stress that \(\scrB_\Gamma\) is not necessarily a simplicial complex.  This
is already evident when \(n=2\): \(\scrB\) is a simple graph, but
\(\scrB_\Gamma\) is not necessarily a simple graph.
This is why we work with the natural \(\Delta\)-complex quotient, which retains
distinct facets even when they have the same set of vertices.

It follows that every vertex, facet, and pointed facet of \(\scrB_\Gamma\) has
a well-defined type valued, respectively, in \(\bbZ/n\bbZ\), the set of
circular partitions, and the set of ordered partitions, as in
Sections~\ref{bt-facet-type}--\ref{bt-pointed-type} (and
Sections~\ref{lattice-facet-type}--\ref{lattice-pointed-type}).  Also, every
pointed \(k\)-facet \((C,v)\) may be represented by
\((v_0,\ldots,v_k)\) as in Sections~\ref{bt-pointed-type}
and~\ref{lattice-pointed-type}.  The geometric and algebraic lengths of
Section~\ref{length} are therefore well-defined on \(\scrB_\Gamma\).

Let \(Y\) be a \(\Delta\)-complex and let \(C\in Y(k)\).  Regarding \(Y\) as
its category of facets, define the \(r\)-facets of the \emph{link} of \(C\) by
\[
\operatorname{Lk}_Y(C)(r)
=
\{\alpha:C\longrightarrow D\mid D\in Y(k+r+1)\}.
\]
Here \(\alpha\) is a specified occurrence of \(C\) as a face of \(D\).  Let
\(j_0<\cdots<j_r\) be the vertices of \(D\) complementary to this occurrence.
The \(i\)-th face of \(\alpha\) is the induced incidence
\(C\to\partial_{j_i}D\) through which \(\alpha\) factors.  These face maps make
\(\operatorname{Lk}_Y(C)\) a \(\Delta\)-complex.  Equivalently, its facets are
the complementary facets associated with the objects \(C\to D\) of the
undercategory \(C/Y\), glued according to the morphisms in \(C/Y\).  When \(Y\)
is a simplicial complex, this agrees with the usual link of \(C\).

\begin{lemma}\label{lem:quotient-links}
Let \(C\) be a facet of \(\scrB_\Gamma\), and choose a lift
\(\widetilde C\) in \(\scrB_{\{1\}}\).  The quotient map induces an isomorphism
\[
\operatorname{Lk}_{\scrB_{\{1\}}}(\widetilde C)
\;\xrightarrow{\ \sim\ }\;
\operatorname{Lk}_{\scrB_\Gamma}(C).
\]
In particular, the link of every proper facet of \(\scrB_\Gamma\) is a
spherical building, and hence has a well-defined opposition relation
\cite{AB}*{4.68 and 4.72}.
\end{lemma}

\begin{proof}
The quotient map \(\scrB_{\{1\}}\to\scrB_\Gamma\) is a covering of
\(\Delta\)-complexes.  Once \(\widetilde C\) is fixed, every coface incidence
of \(C\) lifts uniquely to a coface incidence of \(\widetilde C\), and this
lifting is compatible with all induced face maps.  It therefore gives the
displayed isomorphism of links.  The link of a proper facet in an affine
building is its spherical residue, which is a spherical building
\cite{BT}*{5.1.32}, \cite{tits}*{3.5.4}.  Replacing \(\widetilde C\) by another
lift composes the isomorphism with a type-preserving building automorphism, so
the opposition relation on the quotient link is independent of the lift.
\end{proof}

\subsection{The local successor digraph}
We next define successors directly on the quotient complex, using
Lemma~\ref{lem:quotient-links}.  Let
\[
\dot C=(v_0,v_1,\ldots,v_k),\qquad
\dot C'=(w_0,w_1,\ldots,w_k)
\]
be pointed \(k\)-facets of \(\scrB_\Gamma\).  A \emph{successor arrow} from
\(\dot C\) to \(\dot C'\) consists of specified incidences of a common
\((k-1)\)-facet for which
\[
(w_0,\ldots,w_{k-1})=(v_1,\ldots,v_k)
\]
as ordered vertex occurrences and for which the corresponding vertices in
the link of the common facet are opposite.  Two choices of
face incidences with the same source and target define two different arrows.

Let \(\mathcal D_k(\X)\) be the digraph whose vertices are the pointed
\(k\)-facets of \(\X\) and whose arrows are the successor arrows.  Parallel
arrows and loops are allowed.  We write \(s(e)\) and \(t(e)\) for the source
and target of an arrow \(e\).  Thus \(\mathcal D_k(\X)\), including all
multiplicities needed below, is defined entirely from the local incidence and
link data of \(\scrB_\Gamma\).

\begin{lemma}[Compatibility with the building]\label{lem:local-successor}
The quotient map induces a natural isomorphism of digraphs
\[
\Gamma\backslash\mathcal D_k(\scrB_{\{1\}})\;\xrightarrow{\ \sim\ }
\mathcal D_k(\X).
\]
Consequently, after one lift of its initial vertex has been chosen, every
directed path in \(\mathcal D_k(\X)\) has a unique directed lift to
\(\mathcal D_k(\scrB_{\{1\}})\).
\end{lemma}

\begin{proof}
Lemma~\ref{lem:quotient-links} identifies the relevant links and preserves
their opposition relations.  The covering property also preserves types and
individual face incidences.  Hence successor arrows descend, every successor
arrow lifts uniquely once its source is fixed, and successive arrows in a path
lift compatibly.  This proves both assertions.
\end{proof}

\subsection{\texorpdfstring{\(k\)-geodesics}{k-geodesics}}
A \emph{\(k\)-geodesic} in \(\X\) is a directed path in
\(\mathcal D_k(\X)\); finite, one-sided infinite, and two-sided infinite paths
are all allowed.  When we
speak only of its underlying geodesic, we forget the pointing but retain the
successor incidences traversed by the path.  Lemma~\ref{lem:local-successor}
and Theorem~\ref{twogeodesics} show that the lift of its underlying sequence of
facets is a geometric \(k\)-geodesic in \(\scrB\).

Thus this definition on \(\scrB_\Gamma\) is precisely the combinatorial
realization of the unified notion established in Sections~4--5.  In
particular, when \(\Gamma=\{1\}\), the \(k\)-geodesics just defined are exactly
the combinatorial \(k\)-geodesics of Section~5, upon identifying the
\(\Delta\)-complex \(\scrB_{\{1\}}\) with the simplicial complex \(\scrB\).

For \(k\ge2\), this successor structure gives the orientation used throughout
the paper, and reversing a path need not give another \(k\)-geodesic.
For \(k=1\), the two orientations of an admissible edge path are represented by
different pointed-edge paths.  Only when \(n=2\) is this exactly the usual
non-backtracking path in an ordinary graph; in higher rank the transition is
the stronger, type-sensitive opposition condition in the vertex link.

\medskip
\noindent\textbf{Closed \(k\)-geodesics.}
A closed \(k\)-geodesic of geometric length \(m\) is a
closed directed path
\[
\scrC=(e_0,e_1,\ldots,e_{m-1})
\quad\text{in }\mathcal D_k(\X),
\]
where \(t(e_i)=s(e_{i+1})\), with indices taken modulo \(m\).  Two such paths
are equivalent if their arrow sequences differ by a cyclic shift.  The path is
\emph{primitive} if its cyclic arrow sequence is not a proper power of a
shorter cyclic arrow sequence.  These are entirely local definitions on
\(\X\).

Lemma~\ref{lem:local-successor} also records the monodromy of a closed path.
After a lift of \(s(e_0)\) is chosen, the path has a unique directed lift
\((\dot C_i)_{i\in\bbZ}\), and there is a unique \(\gamma\in\Gamma\) such that
\[
\dot C_{i+m}=\gamma\dot C_i\qquad(i\in\bbZ).
\]
Changing the initial lift conjugates \(\gamma\).  In particular, an \(r\)-th
power of a closed directed path has monodromy an \(r\)-th power.  We will not
use the converse, which need not follow merely from a power decomposition of
the conjugacy class in \(\Gamma\).

\medskip
\noindent\textbf{Lengths.}
For \(\scrC=(e_0,\ldots,e_{m-1})\), the definitions of
Section~\ref{length} give
\[
l_G(\scrC)=m,
\qquad
l_A(\scrC)=\sum_{i=0}^{m-1}l_A\bigl(s(e_i)\bigr).
\]
Both lengths are manifestly invariant under cyclic shift and are defined
locally on \(\X\).

\begin{example}[The two directions when \(n=3\)]\label{ex:pgl3-directions}
Let \(k=1\).  A pointed edge of type \((\lambda,3-\lambda)\) has algebraic
length \(\lambda\), and the successor relation preserves this ordered type.
Thus a closed path of \(m\) such edges has algebraic length \(m\lambda\).  Its
reverse, when regarded as the oppositely pointed closed path, has type
\((3-\lambda,\lambda)\) and algebraic length \(m(3-\lambda)\).  The two
directions are therefore distinct objects and in general have different
algebraic lengths.
\end{example}

\subsection{Geodesic zeta functions}
From now on, assume in addition that \(\Gamma\) is cocompact.  Then \(\X\) is
a finite \(\Delta\)-complex, so the local successor digraphs are finite.
The zeta functions count primitive closed directed paths in the local
successor digraph.

For \(1\le k\le n-1\), we define the \emph{\(k\)-th geodesic zeta function} of
\(\X\) by
\[
Z_k(\X,u)
=
\prod_{[\scrC]}
\left(1-u^{l_A(\scrC)}\right)^{-1},
\]
where \([\scrC]\) ranges over equivalence classes of primitive closed
\(k\)-geodesics in \(\X\).

This zeta function has a natural factorization according to the type of a
\(k\)-geodesic.  If one of its pointed \(k\)-facets has ordered type
\[
(\lambda_0,\ldots,\lambda_{k-1},\lambda_k),
\]
then the formula in Section~\ref{bt-successor-relation} shows that the successor relation cyclically
permutes the first \(k\) entries and leaves the last entry fixed.  Thus the
\(k\)-geodesic has a well-defined type \((a,b)\), where
\[
a=(\!(\lambda_0,\ldots,\lambda_{k-1})\!)
\]
is a circular partition of \(n-b\) into \(k\) parts and
\(b=\lambda_k\) is a positive integer.  Let
\[
Z_{k,a,b}(\X,u)
=
\prod_{\substack{[\scrC]\\\operatorname{type}(\scrC)=(a,b)}}
\left(1-u^{l_A(\scrC)}\right)^{-1}.
\]
Then
\[
Z_k(\X,u)
=
\prod_{(a,b)}Z_{k,a,b}(\X,u),
\]
where the product ranges over all \(1\le b\le n-k\) and all circular
partitions \(a\) of \(n-b\) into \(k\) parts.  When \(n=3\) and \(k=1\), the
two possible values \(b=1,2\) give the two edge zeta functions appearing in
the main formula of \cite{KL}.  The same factorization applies to the twisted
zeta function defined below.

For a closed \(k\)-geodesic \(\scrC\), define
\[
\epsilon(\scrC):=(-1)^{(k+1)l_G(\scrC)}.
\]
We also define the \(\epsilon\)-twisted \(k\)-th geodesic zeta function by
\[
Z_k^\epsilon(\X,u)
=
\prod_{[\scrC]}
\left(1-\epsilon(\scrC)u^{l_A(\scrC)}\right)^{-1}.
\]

For \(k=1\), the two directions are counted separately, as directed cycles in
\(\mathcal D_1(\X)\).  In the rank-one case \(n=2\), one has
\(l_A=l_G\), and this recovers the classical Ihara zeta function.

\subsection{Unramified standard L-functions of
\texorpdfstring{$L^2(\Gamma \backslash G_\ad)$}{L2(Gamma backslash G-ad)}}
\label{lfunction}
For an irreducible unramified representation \(\rho\) of \(G_\ad\) with Satake
parameter \(s(\rho)\), viewed as a semisimple conjugacy class in \(G_\sc(\mathbb{C})\),
define the \(L\)-function of \(\rho\) attached to the standard representation of
\(G_\sc\) (chosen when we fix an orientation on the affine Dynkin diagram
\(\Delta_{\mathrm{aff}}\)) by
\[
L(\rho,u)=\det(I-s(\rho)u)^{-1},
\]
where \(I\) is the identity matrix.
Since \(\Gamma\) is cocompact, \(L^2(\Gamma\backslash G_\ad)\) is a unitary
representation of \(G_\ad\) and decomposes as a direct sum of irreducible
subrepresentations.
The unramified standard \(L\)-function of \(\Gamma\) is defined as
\[
L(\Gamma,u)=\prod_\rho L(\rho,u)^{m(\rho)}.
\]
Here \(\rho\) runs through the irreducible unramified subrepresentations of
this representation, and \(m(\rho)\) denotes the corresponding multiplicity.
This is a finite product: the \(K\)-fixed lines of these constituents lie in
the finite-dimensional space
\(L^2(\Gamma\backslash G_\ad/K)\), whose dimension is
\(\#(\Gamma\backslash G_\ad/K)\).

The main theorem is the following.
\begin{theorem}\label{main}
Let \(\Gamma\) be a discrete torsion-free cocompact subgroup of \(G_\ad\), and
assume that \(\Gamma\) preserves vertex types, equivalently
\(\Gamma\subset\ker(G_\ad\to\Xi_\ad\simeq\mathbb Z/n\mathbb Z)\).  Put
\(\X:=\Gamma\backslash\scrB\). Then
\[
(1-u^n)^{\chi(\X)}\,L(\Gamma,q^{(n-1)/2}u)
=
\prod_{k=1}^{n-1}Z_k^\epsilon(\X,u)^{(-1)^{k+1}},
\]
where \(\chi(\X)\) is the Euler characteristic of the finite \(\Delta\)-complex \(\X\).
Here \(L(\Gamma,u)\) is the unramified standard \(L\)-function defined above.
\end{theorem}

\section{Proof of the Main Theorem}
Throughout this section we write $G=G_{\ad}$.
We prove the main theorem by comparing the Hecke determinant for the unramified
\(L\)-function with the successor-operator determinants for the
\(\epsilon\)-twisted zeta functions.  The remaining determinant computation is
isolated in Section~8.

\subsection{L-functions as Hecke series}\label{s:hecke-series}

Let $K$ be the standard maximal compact subgroup of $G$, namely the stabilizer
of the homothety class of the standard lattice.
The spherical Hecke algebra $H(G//K)$ is generated by the double cosets
\[
A_i
=
K
\begin{pmatrix}
 \pi I_i & \\
 & I_{n-i}
\end{pmatrix}
K,
\qquad
i=0,\ldots,n,
\]
where $A_0=A_n=K$ corresponds to the identity operator.

Each $A_i$ acts on $K$-invariant functions in $L^2(\Gamma\backslash G)$ by
\[
(A_i f)(x)
=
\sum_{g\in A_i/K} f(xg),
\qquad
f\in L^2(\Gamma\backslash G)^K.
\]

Let $(\rho,V)$ be an irreducible unramified representation of $G$ with Satake
parameter
\[
s(\rho)
=
\left[
\begin{pmatrix}
\lambda_1 & & \\
& \ddots & \\
& & \lambda_n
\end{pmatrix}
\right]
\in G_{\sc}(\C),
\]
the complex dual group of \(G_\ad\). Then $V^K$ is one-dimensional.
By the Satake isomorphism \cite{Gr}, the operator $A_i$ acts on $V^K$ by the scalar
\[
q^{i(n-i)/2}
\sum_{1\le m_1<\cdots<m_i\le n}
\lambda_{m_1}\cdots\lambda_{m_i}.
\]

Consequently, the local $L$-factor attached to $\rho$ and the standard
representation can be written as
\begin{align*}
L(\rho,q^{(n-1)/2}u)
&=
\prod_{j=1}^n
\left(1-q^{(n-1)/2}\lambda_j u\right)^{-1}
\\
&=
\det\!\left(
\sum_{i=0}^n (-q^{(i-1)/2}u)^i A_i
\;\middle|\;
V^K
\right)^{-1},
\end{align*}
where both $A_0$ and $A_n$ act as the identity on $V^K$.

\begin{theorem}\label{thm:Hecke-series}
The unramified $L$-function of $\Gamma$ satisfies
\[
L(\Gamma,q^{(n-1)/2}u)
=
\det\!\left(
\sum_{i=0}^n (-q^{(i-1)/2}u)^i A_i
\;\middle|\;
L^2(\Gamma\backslash G/K)
\right)^{-1}.
\]
\end{theorem}

\begin{proof}
We compute the determinant on each unramified isotypic component and then
multiply over the spectrum.
Since $\Gamma$ is cocompact, the representation $L^2(\Gamma\backslash G)$
decomposes discretely as a Hilbert direct sum of irreducible unitary
representations.
Restricting to $K$-fixed vectors yields
\[
L^2(\Gamma\backslash G/K)
=
\bigoplus_{\rho}
m(\rho)\, V_\rho^K,
\]
where $\rho$ ranges over irreducible unramified representations and $m(\rho)$
denotes the multiplicity of $\rho$ in $L^2(\Gamma\backslash G)$.

On each one-dimensional space $V_\rho^K$, the operator
\(\sum_{i=0}^n (-q^{(i-1)/2}u)^i A_i\) acts by the scalar
\[
\det(I-q^{(n-1)/2}s(\rho)u).
\]
Taking determinants over $L^2(\Gamma\backslash G/K)$ and multiplying over all
$\rho$ gives the stated identity.
\end{proof}

\subsection{Cohomology on pointed facets}

The following complex is introduced to compare the two determinant expressions
that arise from the main theorem: degree \(0\) recovers the Hecke operator, while
the higher degrees recover the successor adjacency operators.

Recall that a pointed $k$-facet $(C,v)$ of $\scrB$ can be represented by a chain of
lattices
\[
L_0 \supsetneq L_1 \supsetneq \cdots \supsetneq L_k \supsetneq \pi L_0,
\]
where $L_0$ represents the vertex $v$.

We use the following notation throughout this section.  For a pointed
\(r\)-facet \(a=[a_0,\ldots,a_r]\), extend its lattice chain periodically by
\[
a_{j+r+1}=\pi a_j\qquad(j\in\mathbb Z).
\]
For \(s\le t\le s+r\), write
\(a(s,t)=(a_s,\ldots,a_t)\), and write
\(a(s,\underline{j},t)\) for the same sequence with \(a_j\) omitted.
Square brackets around such a sequence denote the pointed facet that it
represents; notation such as \([b,a(s,t)]\) denotes concatenation.

For \(0\le k\le n-1\), let $C_k(\scrB)$ (resp.\ $C_k(\X)$) denote the space of
$\C(u)$-valued functions on pointed $k$-facets of $\scrB$ (resp.\ $\X$), where
$\C(u)$ is the field of rational functions in one variable $u$.
Set \(C_k(\scrB)=C_k(\X)=0\) for \(k<0\) or \(k\ge n\).
The group $G$ acts naturally on $C_k(\scrB)$, and we identify
\[
C_k(\X)=C_k(\scrB)^\Gamma
\]
as the $\Gamma$-invariant subspace.

Using the notation of Section~\ref{successor}, define coboundary operators,
for \(0\le i\le n-2\),
\[
d_i : C_i(\scrB)\longrightarrow C_{i+1}(\scrB)
\]
by the formula
\[
(d_i f)\bigl[a(0,i+1)\bigr]
=
u^{[a_0:a_1]}\, f\bigl[a(1,i+1)\bigr]
+
\sum_{k=1}^{i+1} (-1)^k
f\bigl[a(0,\underline{k},i+1)\bigr],
\]
where $a=[a(0,i+1)]$ denotes a pointed $(i+1)$-facet.
At the two ends of the complex, set \(d_{-1}=d_{n-1}=0\).

A direct verification shows that
\[
d_i\circ d_{i-1}=0
\qquad(0\le i\le n-1),
\]
so $(C_\bullet(\scrB),d_\bullet)$ is a cochain complex.
Since the action of $G$ commutes with $d_i$, the operators descend to $C_i(\X)$,
and we define the pointed cohomology of $\X$ by
\[
H^i(\X,\C(u))
=
\ker(d_i)/\operatorname{im}(d_{i-1}).
\]

\subsection{Adjacency operators}

For $k=0$, $C_0(\scrB)$ is the space of functions on vertices.
For $i=1,\ldots,n-1$, define the vertex adjacency operators $\hat A_i$ by
\[
(\hat A_i f)([a_0])
=
\sum_{[a_1]:[a_0:a_1]=i} f([a_1]),
\]
where $[a_0:a_1]=\dim_\kappa(a_0/a_1)$.

Identifying the vertex set $\scrB^{(0)}$ with $G/K$, the operators $\hat A_i$
coincide with the Hecke operators $A_i$ introduced in
Section~\ref{s:hecke-series}.
Thus Theorem~\ref{thm:Hecke-series} can be rewritten as follows.

\begin{theorem}\label{thm:Hecke-series-vertex}
We have
\[
L(\Gamma,q^{(n-1)/2}u)
=
\det\!\left(
\sum_{i=0}^n (-q^{(i-1)/2}u)^i \hat A_i
\;\middle|\;
C_0(\X)
\right)^{-1}.
\]
\end{theorem}

For $1\le k\le n-1$, define the combinatorial adjacency operator on the
quotient directly from its local successor digraph by
\[
T_k(u): C_k(\X)\longrightarrow C_k(\X),
\qquad
(T_k f)(x)
=
\sum_{\substack{e\in\mathcal D_k(\X)\\s(e)=x}}
u^{l_A(x)}f\bigl(t(e)\bigr).
\]
Thus parallel successor incidences contribute separate summands.  On the
building, the analogous formula is
\[
T_k(u): C_k(\scrB)\longrightarrow C_k(\scrB)
\]
by
\[
(T_k f)(x)
=
\sum_{y\in\Succ(x)} u^{l_A(x)} f(y),
\]
where the sum runs over all successors $y$ of the pointed $k$-facet $x$.
Here \(\Succ(x)\) denotes the finite set of successor pointed facets of \(x\)
in the building.  The latter operator is $G$-equivariant, and under
Lemma~\ref{lem:local-successor} its restriction to \(\Gamma\)-invariants is the
local quotient operator just defined.  All traces and determinants below are
taken on the finite-dimensional space $C_k(\X)$.

\begin{lemma}
Let \(1\le k\le n-1\), and let \(T_k=T_k(u)\) act on \(C_k(\X)\). For each
$m\ge1$,
\[
\operatorname{tr}(T_k^m)
=
\sum_{\scrC} u^{l_A(\scrC)},
\]
where $\scrC$ runs over all closed directed paths in \(\mathcal D_k(\X)\) of
geometric length $m$, counted with a distinguished initial vertex as in the
trace.
\end{lemma}

On the other hand, for $|u|$ small,
\[
\log Z_k(\X,u)
=
\sum_{[\scrC]\ \mathrm{primitive}}
\sum_{r=1}^\infty \frac{u^{r l_A(\scrC)}}{r}.
\]
Every closed directed path arises as a power of a unique primitive cyclic
arrow sequence, and the number of based representatives of a primitive cycle
equals its geometric length.
Therefore,
\begin{align*}
\log Z_k(\X,u)
&=
\sum_{m=1}^\infty \frac{1}{m}\operatorname{tr}(T_k^m)
=
-\operatorname{tr}\log(I-T_k)
=
\log\det(I-T_k)^{-1}.
\end{align*}

A similar argument yields
\[
Z_k^\epsilon(\X,u)
=
\det\!\left(I-(-1)^{k+1}T_k\right)^{-1}.
\]

\begin{theorem}\label{thm:zeta-determinant}
For \(1\le k\le n-1\), with \(T_k=T_k(u)\) acting on \(C_k(\X)\), the zeta
functions are given by the rational functions
\[
Z_k(\X,u)=\det(I-T_k)^{-1},
\qquad
Z_k^\epsilon(\X,u)=\det\!\left(I-(-1)^{k+1}T_k\right)^{-1}.
\]
\end{theorem}

\subsection{M\"obius functions and Laplacians}

Let $(S,\le)$ be a locally finite partially ordered set.
The M\"obius function $\mu:S\times S\to\mathbb Z$ is the unique function satisfying
\begin{itemize}
\item $\mu(x,x)=1$ for all $x\in S$;
\item $\mu(x,y)=0$ unless $x\ge y$;
\item for $x>y$,
\[
\sum_{x\ge z\ge y}\mu(x,z)=0.
\]
\end{itemize}

When $S$ is the lattice of subspaces of a finite-dimensional vector space $V$
over a finite field $\mathbb F_q$, ordered by inclusion, the M\"obius function is
given by
\[
\mu(x,y)=(-1)^m q^{m(m-1)/2},
\qquad m=\dim(x/y).
\]
We write $\mu(m)$ for this value when only the codimension $m$ matters.

Throughout, we abbreviate $a\supsetneq b$ by $a\supset b$.

\medskip
\noindent\textbf{The operators $R_i$ and $\delta_i$.}
Using the notation of Section~\ref{successor}, for \(1\le i\le n-1\) define
operators
\[
R_i : C_i(\scrB)\longrightarrow C_{i-1}(\scrB)
\]
by
\[
(R_i f)\bigl[a(0,i-1)\bigr]
=
(-1)^i
\sum_{a_0\supset b\supset a_1}
u^{[a_0:b]}\,
\mu\!\left([a_0:b]\right)\,
f\bigl[b,a(1,i)\bigr].
\]

For \(1\le i\le n-1\), define
\[
\delta_i : C_i(\scrB)\longrightarrow C_{i-1}(\scrB)
\]
by
\[
(\delta_i f)\bigl[a(0,i-1)\bigr]
=
\sum_{j=0}^{i-1}
(-1)^{(i+1)j}
u^{[a_0:a_j]}\,
(R_i f)\bigl[a(j,j+i-1)\bigr],
\]
with the convention $a_j=\pi a_{j-i}$ for $j\ge i$.
At the ends set \(\delta_0=\delta_n=0\).

In general, $\delta_i$ is not a coboundary operator; in particular,
$\delta_{i}\circ\delta_{i+1}\neq 0$.

\medskip
\noindent\textbf{Laplacians.}
For \(0\le i\le n-1\), define Laplacian-type operators
\[
\Delta_i
=
d_{i-1}\circ\delta_i+\delta_{i+1}\circ d_i
:
C_i(\scrB)\longrightarrow C_i(\scrB).
\]
Each $\Delta_i$ is a $G$-equivariant cochain endomorphism, and it induces the zero
operator on cohomology:
\[
\Delta_i\equiv 0
\qquad\text{on } H^i(\scrB,\C(u)).
\]
Since this chain-homotopy identity is $G$-equivariant, it descends to
$C_\bullet(\X)$; hence $\Delta_i$ also acts trivially on $H^i(\X,\C(u))$.

Define
\[
\Phi_i=\Delta_i+(1-u^n)\,\mathrm{Id}_{C_i(\scrB)}.
\]
Then $\Phi_i$ is a cochain endomorphism of $C_\bullet(\scrB,\C(u))$, and descends to
an endomorphism of $C_i(\X)$.

\medskip
\noindent\textbf{Determinant identities.}
The key computation is the following.

\begin{theorem}\label{thm:Phi-determinants}
The operators \(\Phi_i\) satisfy the following determinant identities:
\begin{itemize}
\item[(1)]
\[
\det(\Phi_0\mid C_0(\X))
=
\det\!\left(
\sum_{i=0}^n (-q^{(i-1)/2}u)^i \hat A_i
\;\middle|\;
C_0(\X)
\right).
\]

\item[(2)]
For $i\ge1$,
\[
\det(\Phi_i\mid C_i(\X))
=
(1-u^n)^{iV_i}\,
\det\!\left(I-(-1)^{i+1}T_i\mid C_i(\X)\right),
\]
where
\[
V_i=\frac{1}{i+1}\dim C_i(\X)
\]
is the number of $i$-facets of $\X$.
\end{itemize}
\end{theorem}

For convenience, we write
\[
Z_0^\epsilon(\X,u)=L(\Gamma,q^{(n-1)/2}u).
\]

\medskip
\noindent\textbf{Completion of the proof.}
Combining Theorem~\ref{thm:Phi-determinants} with
Theorem~\ref{thm:Hecke-series-vertex} and Theorem~\ref{thm:zeta-determinant}, we
obtain
\[
\prod_{i=0}^{n-1}\det(\Phi_i\mid C_i(\X))^{(-1)^i}
=
\prod_{i=0}^{n-1}
(1-u^n)^{(-1)^i iV_i}\,
Z_i^\epsilon(\X,u)^{(-1)^{i+1}}.
\]

On the other hand, since $\Phi_i$ reduces to the identity at $u=0$, it is a
cochain automorphism with nonzero determinant.  Every \(i\)-facet has
\(i+1\) pointed realizations, so
\(\dim_{\C(u)}C_i(\X)=(i+1)V_i\).  Standard homological algebra and the
Euler--Poincar\'e identity therefore yield
\begin{align*}
\prod_{i=0}^{n-1}\det(\Phi_i\mid C_i(\X))^{(-1)^i}
&=
\prod_{i=0}^{n-1}\det(\Phi_i\mid H^i(\X,\C(u)))^{(-1)^i} \\
&=
\prod_{i=0}^{n-1}\det(1-u^n\mid H^i(\X,\C(u)))^{(-1)^i} \\
&=
(1-u^n)^{\sum_{i=0}^{n-1}(-1)^i(i+1)V_i}.
\end{align*}

Comparing the two expressions and using
\[
\chi(\X)=\sum_{i=0}^{n-1}(-1)^i V_i,
\]
we obtain
\[
\prod_{i=0}^{n-1} Z_i^\epsilon(\X,u)^{(-1)^{i+1}}
=
(1-u^n)^{\chi(\X)},
\]
or equivalently, since $Z_0^\epsilon(\X,u)=L(\Gamma,q^{(n-1)/2}u)$,
\[
(1-u^n)^{\chi(\X)}L(\Gamma,q^{(n-1)/2}u)
=
\prod_{k=1}^{n-1}Z_k^\epsilon(\X,u)^{(-1)^{k+1}}.
\]
This completes the conceptual proof of the main theorem, assuming the determinant
identity for the operators \(\Phi_i\).

\section{Technical proof of Theorem~\ref{thm:Phi-determinants}}
This section proves the determinant identities for the operators \(\Phi_i\).
\subsection{The determinant of \texorpdfstring{$\Phi_0$}{Phi-0}}

We show that
\[
\det(\Phi_0\mid C_0(\X))
=
\det\!\left(
\sum_{k=0}^n (-q^{(k-1)/2}u)^k\,\hat A_k
\;\middle|\;
C_0(\X)
\right).
\]

Recall that
\[
\begin{aligned}
(d_0 f)[a(0,1)]
&=
u^{[a_0:a_1]}f[a_1]-f[a_0],
\\
(R_1 g)[a_0]
&=
-\!\!\sum_{a_0\supset b\supset \pi a_0}
u^{[a_0:b]}\mu([a_0:b])\,g[b,\pi a_0].
\end{aligned}
\]
and that $\Phi_0=\delta_1 d_0+(1-u^n)\mathrm{Id}$ with $\delta_1=R_1$.

A direct computation gives
\begin{align*}
(\Phi_0 f)[a_0]
&=
(1-u^n)f[a_0]
\\
&\quad
-\sum_{a_0\supset b\supset \pi a_0}
\mu([a_0:b])
\Bigl(
u^{[a_0:b]+[b:\pi a_0]}f[\pi a_0]
-
u^{[a_0:b]}f[b]
\Bigr).
\end{align*}
Since $f[\pi a_0]=f[a_0]$ and $[a_0:b]+[b:\pi a_0]=n$, this simplifies to
\[
(\Phi_0 f)[a_0]
=
\Bigl(1+\mu(n)u^n\Bigr)f[a_0]
+
\sum_{a_0\supset b\supset \pi a_0}
\mu([a_0:b])u^{[a_0:b]}f[b].
\]

Grouping terms by $i=[a_0:b]$ yields
\[
(\Phi_0 f)[a_0]
=
\sum_{i=0}^n \mu(i)u^i\,\hat A_i f[a_0],
\]
where $\hat A_0=\hat A_n=\mathrm{Id}$. Using $\mu(i)=(-1)^i q^{i(i-1)/2}$, we obtain
\begin{align*}
\Phi_0
&=
\sum_{i=0}^n (-1)^i q^{i(i-1)/2}u^i\,\hat A_i
\\
&=
\sum_{i=0}^n (-q^{(i-1)/2}u)^i\,\hat A_i,
\end{align*}
which proves the claimed determinant identity.

\subsection{The determinant of \texorpdfstring{$\Phi_i$}{Phi-i} for
\texorpdfstring{$i>0$}{i greater than 0}}

For $i>0$, we prove
\[
\det(\Phi_i\mid C_i(\X))
=
(1-u^n)^{iV_i}\,
\det\!\left(I-(-1)^{i+1}T_i \mid C_i(\X)\right),
\]
where $V_i$ denotes the number of $i$-facets of $\X$.

To this end, we introduce two endomorphisms of $C_i(\X)$:
\[
(W_i f)[a(0,i)]
=
f[a(0,i)]
+
(-1)^{i+1}u^{[a_0:a_1]}f[a(1,i\!+\!1)],
\]
and
\[
(Q_i f)[a(0,i)]
=
\sum_{a_0 \supseteq c \supset a_1}
u^{[a_0:c]} f[c,a(1,i)].
\]

\begin{theorem}\label{QW}
For \(i\ge1\), the operators \(W_i,Q_i\) on \(C_i(\X)\) satisfy:
\begin{enumerate}
\item $\det(W_i\mid C_i(\X))=(1-u^n)^{V_i}$.
\item $\det(Q_i\mid C_i(\X))=1$.
\end{enumerate}
\end{theorem}

\begin{proof}
\textup{(1)}
Fix an $i$-facet of $\X$. The subspace spanned by functions supported on its
$(i+1)$ pointed realizations is $W_i$-invariant.
On this subspace, the nontrivial part of $W_i$ is a cyclic weighted shift whose
weights have product
\[
(-1)^{(i+1)^2}u^{[a_0:a_1]+\cdots+[a_i:\pi a_0]}=(-1)^{i+1}u^n,
\]
and for a cyclic weighted shift on $i+1$ basis vectors with weight product
$P$, the determinant of $I$ plus the shift is $1-(-1)^{i+1}P$.  Hence the
determinant of
$W_i$ on this block equals $1-u^n$.
Since there are $V_i$ such independent blocks, we obtain
\[
\det(W_i\mid C_i(\X))=(1-u^n)^{V_i}.
\]

\textup{(2)}
For \(Q_i\), we order pointed facets by the associated ordered partition.
Define a partial order on ordered partitions
$(\lambda_0,\dots,\lambda_i)$ by
\[
(\lambda_0,\dots,\lambda_i)\ge(\lambda_0',\dots,\lambda_i')
\iff
\sum_{j=0}^k \lambda_j \ge \sum_{j=0}^k \lambda_j'
\quad\text{for all }k.
\]
For a pointed $i$-facet $[a(0,i)]$, the non-diagonal summands
$f[c,a(1,i)]$ appearing in $(Q_i f)[a(0,i)]$ correspond to strictly smaller
partitions in this order.
After choosing any linear extension of this partial order, $Q_i$ is upper
triangular with ones on the diagonal, hence unipotent, and
\[
\det(Q_i\mid C_i(\X))=1.
\]
\end{proof}

By M\"obius inversion on the lattice of subspaces, we obtain the following
summation identities involving the operators $Q_i$ and $W_i$.

\begin{lemma}\label{lemma3}
Let \(i\ge1\), let \(f\in C_i(\X)\), and let \([a(0,i)]\) be a pointed
\(i\)-facet. Then
\[
\sum_{a_0 \supseteq b \supset a_1}
\mu([a_0:b])\,u^{[a_0:b]}\,
(Q_i f)[b,a(1,i)]
=
f[a(0,i)].
\]
\end{lemma}

\begin{proof}
By definition of $Q_i$,
we expand the left-hand side and then collect terms according to the
intermediate lattice \(c\).
\begin{align*}
&\sum_{a_0 \supseteq b \supset a_1}
\mu([a_0:b])u^{[a_0:b]}(Q_i f)[b,a(1,i)]
\\
&\quad =
\sum_{a_0 \supseteq b \supset a_1}
\sum_{b \supseteq c \supset a_1}
\mu([a_0:b])u^{[a_0:c]} f[c,a(1,i)] \\
&\quad =
\sum_{a_0 \supseteq c \supset a_1}
\left(\sum_{a_0 \supseteq b \supseteq c}\mu([a_0:b])\right)
u^{[a_0:c]} f[c,a(1,i)].
\end{align*}
By the defining property of the M\"obius function,
\[
\sum_{a_0 \supseteq b \supseteq c}\mu([a_0:b])
=
\begin{cases}
1, & c=a_0,\\
0, & c\neq a_0.
\end{cases}
\]
Hence only the term $c=a_0$ survives, yielding $f[a(0,i)]$.
\end{proof}

\begin{lemma}\label{lemma4}
Let \(i\ge1\), let \(f\in C_i(\X)\), and let \([a(0,i+1)]\) be a pointed
\((i+1)\)-facet. Then
\[
\sum_{\substack{a_0 \supset b \supset a_2\\ b\not\supseteq a_1}}
\mu([a_0:b])u^{[a_0:b]}\,
(Q_i f)[b,a(2,i\!+\!1)]
=
-\!\!\sum_{\substack{a_0 \supset b \supset a_2\\ b+a_1=a_0}}
u^{[a_0:b]} f[b,a(2,i\!+\!1)].
\]
\end{lemma}

\begin{proof}
Expanding $Q_i$ gives
\begin{align*}
&\sum_{\substack{a_0 \supset b \supset a_2\\ b\not\supseteq a_1}}
\mu([a_0:b])u^{[a_0:b]}(Q_i f)[b,a(2,i\!+\!1)] \\
&\qquad=
\sum_{\substack{a_0 \supset b \supset a_2\\ b\not\supseteq a_1}}
\sum_{b \supseteq c \supset a_2}
\mu([a_0:b])u^{[a_0:c]} f[c,a(2,i\!+\!1)] \\
&\qquad=
\sum_{a_0 \supset c \supset a_2}
\left(
\sum_{\substack{a_0 \supset b \supseteq c\\ b\not\supseteq a_1}}
\mu([a_0:b])
\right)
u^{[a_0:c]} f[c,a(2,i\!+\!1)].
\end{align*}
The condition \(b\not\supseteq a_1\) is handled by subtracting the terms with
\(b\supseteq c+a_1\).  The inner sum decomposes as
\[
\sum_{a_0 \supseteq b \supseteq c}\mu([a_0:b])
-
\sum_{a_0 \supseteq b \supseteq (c+a_1)}\mu([a_0:b]).
\]
By M\"obius inversion,
\[
\sum_{a_0 \supseteq b \supseteq c}\mu([a_0:b])=0,
\qquad
\sum_{a_0 \supseteq b \supseteq (c+a_1)}\mu([a_0:b])
=
\begin{cases}
1,& c+a_1=a_0,\\
0,& c+a_1\neq a_0.
\end{cases}
\]
Here $c\neq a_0$, so the first inclusive M\"obius sum is zero.  Therefore the
coefficient equals $-1$ precisely when $c+a_1=a_0$, and $0$ otherwise, yielding
the stated identity.
\end{proof}

\begin{lemma}\label{lemma1}
For $i\ge1$, \(f\in C_{i-1}(\X)\), and any pointed \(i\)-facet \([a(0,i)]\),
we have
\[
(W_i d_{i-1})f[a(0,i)]
=
\sum_{k=1}^{i}(-1)^k\, W_{i-1}f[a(0,\underline{k},i)].
\]
\end{lemma}

\begin{proof}
By definition of $W_i$,
\begin{align*}
(W_i d_{i-1})f[a(0,i)]
&= d_{i-1}f[a(0,i)]
   +(-1)^{i+1}u^{[a_0:a_1]} d_{i-1}f[a(1,i+1)].
\end{align*}
Expanding both $d_{i-1}$ terms and regrouping yields
\begin{align*}
&u^{[a_0:a_1]}f[a(1,i)]
+\sum_{k=1}^{i}(-1)^k f[a(0,\underline{k},i)]
\\
&\quad
+(-1)^{i+1}u^{[a_0:a_1]}
\Bigl(
u^{[a_1:a_2]}f[a(2,i+1)]
+\sum_{k=1}^{i}(-1)^k f[a(1,\underline{k+1},i+1)]
\Bigr),
\end{align*}
Grouping the terms with the same omitted index gives exactly
\[
\sum_{k=1}^{i}(-1)^k\, W_{i-1}f[a(0,\underline{k},i)].
\]
\end{proof}

\begin{lemma}\label{lemma2}
For $i\ge1$, \(f\in C_{i+1}(\X)\), and any pointed \(i\)-facet \([a(0,i)]\),
we have
\[
(W_i \delta_{i+1})f[a(0,i)]
=
(1-u^n)\, R_{i+1}f[a(0,i)].
\]
\end{lemma}

\begin{proof}
Using the definitions of $W_i$ and $\delta_{i+1}$,
\begin{align*}
(W_i \delta_{i+1})f[a(0,i)]
&=
\delta_{i+1}f[a(0,i)]
+(-1)^{i+1}u^{[a_0:a_1]}\delta_{i+1}f[a(1,i+1)]\\
&=
\sum_{j=0}^{i}(-1)^{ij}u^{[a_0:a_j]}R_{i+1}f[a(j,j+i)]
\\
&\quad
+\sum_{j=1}^{i+1}(-1)^{ij+1}u^{[a_0:a_1]+[a_1:a_j]}R_{i+1}f[a(j,j+i)].
\end{align*}
The two displayed sums have the same terms with opposite signs after shifting
the index, except for the boundary terms whose total exponent is
\(n=[a_0:\pi a_0]\).  These remaining terms give
\((1-u^n)R_{i+1}f[a(0,i)]\).
\end{proof}

\begin{lemma}\label{main2}
For $i\ge1$, \(f\in C_i(\X)\), and any pointed \(i\)-facet \([a(0,i)]\), we have
\[
\frac{1}{1-u^n}\, W_i\Phi_i Q_i f[a(0,i)]
=
f[a(0,i)]
+(-1)^i
\!\!\sum_{\substack{a_0 \supset b \supset a_2\\ b+a_1=a_0}}
u^{[a_0:b]} f[b,a(2,i+1)].
\]
\end{lemma}

\begin{proof}
By Lemmas~\ref{lemma1} and~\ref{lemma2},
\[
\frac{1}{1-u^n}(W_i d_{i-1}\delta_i)f[a(0,i)]
=
\sum_{k=1}^{i}(-1)^k R_i f[a(0,\underline{k},i)],
\]
and
\[
\frac{1}{1-u^n}(W_i\delta_{i+1}d_i)f[a(0,i)]
=
R_{i+1}d_if[a(0,i)].
\]
Thus
\begin{align*}
\frac{1}{1-u^n}W_i\Phi_i f[a(0,i)]
&=W_i f[a(0,i)]\\
&\quad+\sum_{k=1}^{i}(-1)^kR_i f[a(0,\underline{k},i)]\\
&\quad+R_{i+1}d_i f[a(0,i)].
\end{align*}
To display the cancellation in this expression, expand its last term as
\begin{align*}
R_{i+1}d_i f[a(0,i)]
&=
(-1)^{i+1}
\sum_{a_0\supset b\supset a_1}
\mu([a_0:b])u^{[a_0:b]}
\\
&\qquad\mathrel{}\times
\left(
u^{[b:a_1]}f[a(1,i+1)]
+\sum_{h=1}^{i+1}(-1)^h
f[b,a(1,\underline{h},i+1)]
\right).
\end{align*}
The term with \(h=i+1\), together with the first summand of \(W_i f\),
gives the first sum below, including its endpoint \(b=a_0\).  For
\(2\le h\le i\), the term omitting \(a_h\) cancels the corresponding term
in the middle sum above.
For \(h=1\), the same cancellation removes the part of the \(k=1\) term
with \(b\supset a_1\); the weighted first-face term and the second summand
of \(W_i f\) remove its remaining term \(b=a_1\).  Thus the uncancelled
part is characterized by \(b\not\supseteq a_1\).
Consequently,
\begin{align*}
\frac{1}{1-u^n}W_i\Phi_if[a(0,i)]
&=
\sum_{a_0 \supseteq b \supset a_1}
\mu([a_0:b])u^{[a_0:b]}f[b,a(1,i)]\\
&\quad
-(-1)^i\!\!\sum_{\substack{a_0 \supset b \supset a_2\\ b\not\supseteq a_1}}
\mu([a_0:b])u^{[a_0:b]}f[b,a(2,i+1)].
\end{align*}
Now replace $f$ in this identity by $Q_i f$.  Lemma~\ref{lemma3} evaluates the
first sum as $f[a(0,i)]$, while Lemma~\ref{lemma4} evaluates the second sum.
Thus
\[
\frac{1}{1-u^n}\, W_i\Phi_i Q_i f[a(0,i)]
=
f[a(0,i)]
+(-1)^i
\!\!\sum_{\substack{a_0 \supset b \supset a_2\\ b+a_1=a_0}}
u^{[a_0:b]} f[b,a(2,i+1)],
\]
as claimed.
\end{proof}

By Section~\ref{successor}, $b(0,i)$ is a successor of $a(0,i)$ if
\[
a_1=b_0,\ a_2=b_1,\ \ldots,\ a_i=b_{i-1},
\qquad
a_{i+1}\cap b_i=a_{i+2},
\qquad
a_{i+1}+b_i=a_i.
\]
Thus $T_i$ can be written as
\begin{align}
(T_i f)[a(0,i)]
&=
\sum_{\substack{a_i \supset b \supset a_{i+2}\\
b+a_{i+1}=a_i\\
b\cap a_{i+1}=a_{i+2}}}
f[a(1,i),b]\,u^{[a_0:a_1]}
\nonumber\\
&=
\sum_{\substack{a_i \supset b \supset a_{i+2}\\
b+a_{i+1}=a_i\\
b\cap a_{i+1}=a_{i+2}}}
f[a(1,i),b]\,u^{[a_i:b]}.
\label{eq71}
\end{align}

Set $ (J f)[a(0,i)] = f[a(1,i+1)]$ and let $\Lambda_i = \frac{1}{1-u^n}W_i \Phi_i Q_i -I$ for short. Then

\begin{align} 
(J^{-1} \Lambda_i J f)[a(0,i)]
&= (\Lambda_i J f)[a(-1,i-1)] \nonumber \\
&= (-1)^i
\sum_{\substack{a_{-1} \supset b \supset a_{1}\\ b+a_{0}=a_{-1}}}
(Jf)[b,a(1,i)]\,u^{[a_{-1}:b]} \nonumber \\
&= (-1)^i
\sum_{\substack{a_{-1} \supset b \supset a_{1}\\ b+a_{0}=a_{-1}}}
f[a(1,i),\pi b]\,u^{[a_{-1}:b]} \nonumber \\
&= (-1)^i
\sum_{\substack{a_{i} \supset b' \supset a_{i+2}\\ b'+a_{i+1}=a_{i}}}
f[a(1,i),b']\,u^{[a_i:b']}.
\label{eq72}
\end{align}
Here we set $b'=\pi b$ and use the relation $a_{j+i+1} = \pi a_j$ in the last equality.
\begin{theorem}
For \(i\ge1\), let \(V_i\) be the number of \(i\)-facets of \(\X\). Then
\[
\det(\Phi_i\mid C_i(\X))
=
(1-u^n)^{iV_i}
\det\!\left(I-(-1)^{i+1}T_i\mid C_i(\X)\right).
\]
\end{theorem}
\begin{proof}
First we reduce \(\Phi_i\) to the auxiliary operator \(\Sigma_i\).
Let 
\[
\begin{aligned}
\Sigma_i
&= (-1)^i J^{-1} \Lambda_i J \\
&= (-1)^i
\left(\frac{1}{1-u^n}J^{-1}W_i \Phi_i Q_i J -I\right).
\end{aligned}
\]
Then 
\[
\Phi_i
=
(1-u^n)W_i^{-1}J
\bigl(I+(-1)^i\Sigma_i\bigr)
J^{-1}Q_i^{-1}.
\]
By Theorem~\ref{QW}, $\det(Q_i)=1$ and
$\det(W_i)=(1-u^n)^{V_i}$.  Then
\begin{align*}
\det(\Phi_i) &= 
\det\bigg((1-u^n)  W_i^{-1}J(I + (-1)^i \Sigma_i)J^{-1}Q_i^{-1}\bigg)\\
&= (1-u^n)^{(i+1) V_{i}}
\det\bigg( W_i^{-1}J(I + (-1)^i\Sigma_i)J^{-1}Q_i^{-1}\bigg) \\
&= (1-u^n)^{i V_{i}} \det(I+(-1)^i \Sigma_i).
\end{align*}
Here $V_i$ is the number of $i$-facets of $\X$ and $(i+1)V_i$ is the size of the matrix $\Phi_i$.

It remains to compare \(\Sigma_i\) with \(T_i\) by traces.  We first show that
\[
\det\!\left(I-(-1)^{i+1} \Sigma_i\right)
=
\det\!\left(I-(-1)^{i+1}T_i\mid C_i(\X)\right).
\]

Both \(\Sigma_i\) and \(T_i\) have entries in \(u\C[[u]]\), the maximal ideal
of \(\C[[u]]\).  For any \(m\times m\) matrix \(M\) with entries in this ideal,
the formal identity
\[
\det(I_m - M)
=
\exp\left(-\sum_{r=1}^\infty \frac{\tr(M^r)}{r}\right)
\]
holds.  Therefore it suffices to show that \(\Sigma_i^r\) and \(T_i^r\) have
the same trace for every \(r\ge1\).  By
Equations~\eqref{eq71} and~\eqref{eq72}, we have
\[
\begin{aligned}
(T_i f)[a(0,i)]
&=
\sum_{\substack{a_i \supset b \supset a_{i+2}\\
b+a_{i+1}=a_i\\
b\cap a_{i+1}=a_{i+2}}}
f[a(1,i),b]\,u^{[a_i:b]}
\end{aligned}
\]
and
\[
\begin{aligned}
(\Sigma_i f)[a(0,i)]
&=
\sum_{\substack{a_{i} \supset b \supset a_{i+2}\\ b+a_{i+1}=a_{i}}}
f[a(1,i),b]\,u^{[a_i:b]}.
\end{aligned}
\]

Now given two pointed $i$-facets $a(0,i)$ and $a'(0,i)$, 
\begin{itemize}
\item denote by  $a(0,i)\rightarrow a'(0,i)$ the condition that $a(1,i)=a'(0,i-1)$ and $a'_i +a_{i+1}=a_i$;
\item denote by  $a(0,i)\Rightarrow a'(0,i)$ the same condition together with the extra condition  $a'_i \cap a_{i+1}=a_{i+2}$. 
\end{itemize}
Under these notations, we have 
\[
(T_i f)[a(0,i)]
=
\sum_{a(0,i)\Rightarrow a'(0,i)}
f[a'(0,i)]\,u^{[a_i:a'_i]}
\]
and
\[
(\Sigma_i f)[a(0,i)]
=
\sum_{a(0,i)\rightarrow a'(0,i)}
f[a'(0,i)]\,u^{[a_i:a'_i]}.
\]

We now compare closed paths contributing to the two traces.
Lifting closed paths in the finite quotient to $\scrB$, the trace of $(T_i)^r$
is the weighted sum over the following sequences of pointed $i$-facets modulo
the action of $\Gamma$:
\[
\begin{gathered}
\left(a^{(0)},a^{(1)},\ldots,a^{(r)}\right)
\text{ with } a^{(j)} \Rightarrow a^{(j+1)} \text{ for all }j,\\
\text{and }a^{(r)}=\gamma a^{(0)}
\text{ for some }\gamma\in\Gamma.
\end{gathered}
\]
Similarly, the trace of $(\Sigma_i)^r$ is the weighted sum over the following
sequences of pointed $i$-facets modulo the action of $\Gamma$:
\[
\begin{gathered}
\left(a^{(0)},a^{(1)},\ldots,a^{(r)}\right)
\text{ with } a^{(j)} \rightarrow a^{(j+1)} \text{ for all }j,\\
\text{and }a^{(r)}=\gamma a^{(0)}
\text{ for some }\gamma\in\Gamma.
\end{gathered}
\]
To complete the proof, let us show the above two conditions are equivalent.
Suppose \(a^{(j)}\rightarrow a^{(j+1)}\) for all \(j\).  We have to show that
\(a^{(j+1)}_i\cap a^{(j)}_{i+1}
=a^{(j)}_{i+2}=a^{(j+1)}_{i+1}\) for all \(j\).
Since \(a^{(j+1)}_i+a^{(j)}_{i+1}=a^{(j)}_i\), we have
\[
\dim_\kappa a^{(j+1)}_i/a^{(j+1)}_{i+1}
=
\dim_\kappa a^{(j+1)}_i/a^{(j)}_{i+2}
\geq
\dim_\kappa a^{(j)}_i/a^{(j)}_{i+1}
\qquad(0\leq j\leq r-1).
\]
The intersection condition is equivalent to equality in this inequality.

On the other hand, \(a^{(r)}=\gamma a^{(0)}\) for some
\(\gamma\in\Gamma\), and hence
\[
\dim_\kappa a^{(r)}_i/a^{(r)}_{i+1}
=
\dim_\kappa \gamma(a^{(0)}_i)/\gamma(a^{(0)}_{i+1})
=
\dim_\kappa a^{(0)}_i/a^{(0)}_{i+1}.
\]
It follows that
\[
\dim_\kappa a^{(0)}_i/a^{(0)}_{i+1}
=
\dim_\kappa a^{(1)}_i/a^{(1)}_{i+1}
=\cdots=
\dim_\kappa a^{(r)}_i/a^{(r)}_{i+1}.
\]
Thus every closed sequence satisfying \(\rightarrow\) also satisfies
\(\Rightarrow\).
The converse is immediate because \(\Rightarrow\) is defined by adding the
intersection condition to \(\rightarrow\), and the weights in the two trace sums
agree.
Therefore $\tr(\Sigma_i^r)=\tr(T_i^r)$ for all $r$, completing the proof.

\end{proof}

\section{Central division algebras}

We now show that the preceding theory extends from \(F\) to a central
division algebra over \(F\).  The purpose of this section is both to state the
result precisely and to record the few changes needed in the proof.

\subsection{The building and its lattice model}

Let \(D\) be a central division algebra over \(F\), of degree \(d\), so that
\(\dim_F D=d^2\).  Let
\[
v_D:D^\times\longrightarrow\bbZ
\]
be the normalized valuation, let \(\ringO_D\) be its valuation ring, and
choose a uniformizer \(\Pi\).  Thus the unique maximal two-sided ideal is
\[
\mathfrak p_D=\Pi\ringO_D=\ringO_D\Pi.
\]
The residue division ring \(\kappa_D=\ringO_D/\mathfrak p_D\) is a finite
field.  If \(q=\#\kappa\), then the standard structure theory of division
algebras over local fields gives
\begin{equation}\label{eq:division-residue}
Q:=\#\kappa_D=q^d.
\end{equation}
See, for example, \cite{Sat}*{Section~8.1}.  We use \(Q\), rather than \(q\),
as the residue parameter throughout this section.

Let \(V\) be an \(n\)-dimensional right \(D\)-vector space.  A
\(D\)-lattice in \(V\) is a finitely generated right \(\ringO_D\)-submodule
that spans \(V\) over \(D\).  Two \(D\)-lattices are homothetic if they
differ by right multiplication by an element of \(D^\times\).  The vertices
of the reduced Bruhat--Tits building \(\scrB_D\) of
\[
G_D=\PGL_n(D):=\GL_D(V)/F^\times
\]
are the homothety classes of \(D\)-lattices.  A pointed \(k\)-facet is
represented by a chain
\begin{equation}\label{eq:division-chain}
L_0\supsetneq L_1\supsetneq\cdots\supsetneq L_k\supsetneq L_0\Pi.
\end{equation}
Its ordered type is
\[
(\lambda_0,\ldots,\lambda_k),\qquad
\lambda_i=\dim_{\kappa_D}(L_i/L_{i+1}),
\]
where \(L_{k+1}=L_0\Pi\).  Since \(L_0/L_0\Pi\) has dimension \(n\) over
\(\kappa_D\), one has \(\lambda_0+\cdots+\lambda_k=n\).  In particular, the
types, circular partitions, and ordered partitions used in Sections~2--3 are
unchanged.  The type of a vertex is again valued in \(\bbZ/n\bbZ\); in the
lattice model it is given by relative \(\ringO_D\)-length modulo \(n\), or
equivalently by the valuation of the Dieudonn\'e determinant modulo \(n\).

The building \(\scrB_D\) is an affine building of type
\(\widetilde A_{n-1}\).  Its apartments have the same affine Coxeter
complexes as those of \(\scrB\), and its spherical residues are the
corresponding flag buildings over \(\kappa_D\).  Thus the definitions of
opposite neighboring facets and successors in Section~2 apply without
change.  In the lattice model, two opposite neighboring facets are again
characterized by the sum and intersection conditions of Section~3, now for
right \(\ringO_D\)-lattices.  The geometric and algebraic lengths are
\[
l_G(\dot C)=1,\qquad l_A(\dot C)=\lambda_0.
\]

\subsection{Geometric and combinatorial geodesics}

\begin{proposition}\label{prop:division-geodesics}
For the building \(\scrB_D\), geometric \(k\)-geodesics and combinatorial
\(k\)-geodesics coincide in the sense of Theorem~\ref{twogeodesics}.  Hence
\(k\)-geodesics in \(\scrB_D\), together with their compatible pointings,
ordered types, and two length functions, are defined exactly as in
Sections~4--5.
\end{proposition}

\begin{proof}
The arguments of Section~4 take place in apartments and use only the affine
Coxeter complex of type \(\widetilde A_{n-1}\), together with convexity and
opposition in spherical links.  They therefore apply verbatim to
\(\scrB_D\).

For Section~5, replace the elementary-divisor theorem over \(\ringO\) by its
version for right lattices over the noncommutative discrete valuation ring
\(\ringO_D\); see \cite{Sat}*{Section~8.1}.  Every quotient occurring between
\(L\) and \(L\Pi\) is a vector space over \(\kappa_D\).  The block conditions
still depend only on \(v_D\): the ideal \(\mathfrak p_D\) is two-sided, and
conjugation by \(D^\times\) preserves \(\ringO_D\).  Consequently, the block
stabilizer calculations and the transitivity on complements reduce to the
same linear algebra over \(\kappa_D\).  These are the only field-dependent
ingredients in the proof of Theorem~\ref{twogeodesics}.
\end{proof}

Now let \(\Gamma\) be a discrete type-preserving subgroup of \(G_D\) acting
freely on \(\scrB_D\).  Ordering the vertex types by
\(0<1<\cdots<n-1\) makes \(\scrB_D\) into a \(\Delta\)-complex, and we put
\[
\scrB_{D,\Gamma}=\Gamma\backslash(\scrB_D)_\Delta.
\]
The proof of Lemma~\ref{lem:quotient-links} is purely a covering argument, so
every link in \(\scrB_{D,\Gamma}\) is the corresponding spherical residue of
\(\scrB_D\).  Consequently the local successor digraph
\(\mathcal D_k(\scrB_{D,\Gamma})\), and hence \(k\)-geodesics on
\(\scrB_{D,\Gamma}\), are defined without any finiteness assumption.  If
\(\Gamma\) is cocompact, the complex is finite and the Euler products
\(Z_k(\scrB_{D,\Gamma},u)\) and \(Z_k^\epsilon(\scrB_{D,\Gamma},u)\) are
defined exactly as in Section~6.

\subsection{The relative spherical \texorpdfstring{\(L\)-function}{L-function}}

Let \(K_D\) be the stabilizer in \(G_D\) of the homothety class of a standard
lattice.  For \(1\le i\le n-1\), define the vertex adjacency operator
\(A_{i,D}\) by
\[
(A_{i,D}f)([L])
=
\sum_{\substack{L\supset L'\supset L\Pi\\
\dim_{\kappa_D}(L/L')=i}}
f([L']),
\]
and put \(A_{0,D}=A_{n,D}=I\).  These are the fundamental operators in the
spherical Hecke algebra of \((G_D,K_D)\); equivalently, \(A_{i,D}\) is the
operator attached to the double coset
\[
K_D\operatorname{diag}(\Pi I_i,I_{n-i})K_D.
\]

Let \(\rho\) be an irreducible \(K_D\)-spherical representation of \(G_D\).
Its space of \(K_D\)-fixed vectors is one-dimensional.  Satake's calculation
for \(\PGL_n(D)\) associates with \(\rho\) a relative Satake parameter
\[
s_D(\rho)=[\diag(\lambda_1,\ldots,\lambda_n)],
\qquad \lambda_1\cdots\lambda_n=1,
\]
well-defined up to permutation, for which \(A_{i,D}\) acts by
\begin{equation}\label{eq:division-satake}
Q^{i(n-i)/2}
\sum_{1\le m_1<\cdots<m_i\le n}
\lambda_{m_1}\cdots\lambda_{m_i}.
\end{equation}
See \cite{Sat}*{Section~8.4 and Appendix~I}.  Although \cite{Sat} is written
for \(p\)-adic number fields, the calculation uses only elementary-divisor
theory and the cardinality of the residue field, and therefore applies
unchanged to equal-characteristic local fields.  We define the relative
spherical \(L\)-factor by
\[
L_D(\rho,u)=\det(I-s_D(\rho)u)^{-1}.
\]
For a discrete cocompact subgroup \(\Gamma\), let
\[
L_D(\Gamma,u)=\prod_\rho L_D(\rho,u)^{m(\rho)},
\]
where the product runs through the irreducible spherical constituents of
\(L^2(\Gamma\backslash G_D)\), with sphericality taken relative to \(K_D\).
This product is finite because
\(L^2(\Gamma\backslash G_D/K_D)\) is finite-dimensional.

The adjective \emph{relative} is included to make the normalization clear.
When \(D=F\), this is exactly the unramified \(L\)-function of Section~6.7.
For \(d>1\), the algebraic group \(G_D\) is an inner form of
\(\PGL_{nd}\), whereas the building and its spherical Hecke polynomial have
relative rank \(n-1\) and degree \(n\).  The factor defined above is the
degree-\(n\) spherical factor naturally detected by the building.

Equation~\eqref{eq:division-satake} gives, exactly as in
Theorem~\ref{thm:Hecke-series},
\begin{equation}\label{eq:division-hecke-determinant}
L_D(\Gamma,Q^{(n-1)/2}u)
=
\det\!\left(
\sum_{i=0}^n(-Q^{(i-1)/2}u)^i A_{i,D}
\;\middle|\;
L^2(\Gamma\backslash G_D/K_D)
\right)^{-1}.
\end{equation}

\subsection{The division-algebra identity}

\begin{theorem}\label{thm:division-main}
Let \(D\) be a central division algebra over \(F\), let
\(Q=\#\kappa_D\), and let \(\scrB_D\) be the Bruhat--Tits building of
\(G_D=\PGL_n(D)\).  Suppose that \(\Gamma\) is a discrete torsion-free
cocompact subgroup of \(G_D\) that preserves vertex types, and put
\(\scrB_{D,\Gamma}=\Gamma\backslash\scrB_D\).  Then
\[
(1-u^n)^{\chi(\scrB_{D,\Gamma})}
L_D(\Gamma,Q^{(n-1)/2}u)
=
\prod_{k=1}^{n-1}
Z_k^\epsilon(\scrB_{D,\Gamma},u)^{(-1)^{k+1}}.
\]
In particular, Theorem~\ref{main} is the case \(D=F\).
\end{theorem}

\begin{proof}
By Proposition~\ref{prop:division-geodesics} and the discussion following it,
all geodesic zeta functions and successor operators used in Sections~6--8 are
available for \(\scrB_{D,\Gamma}\).  It remains only to check the two numerical
inputs in the determinant calculation.

First, every interval
\[
L\supseteq L'\supseteq L\Pi
\]
is the lattice of subspaces of the \(n\)-dimensional \(\kappa_D\)-vector
space \(L/L\Pi\).  Its M\"obius function is therefore
\[
\mu(r)=(-1)^rQ^{r(r-1)/2}.
\]
Thus the proof of Theorem~\ref{thm:Phi-determinants}, including all the
sum-and-intersection identities in Section~8, is unchanged after replacing
\(q\) by \(Q\).  It gives
\[
\det(\Phi_0)
=
\det\!\left(\sum_{i=0}^n(-Q^{(i-1)/2}u)^iA_{i,D}\right)
\]
and, for \(i\ge1\),
\[
\det(\Phi_i)
=(1-u^n)^{iV_i}\det\!\left(I-(-1)^{i+1}T_i\right).
\]

Second, Equation~\eqref{eq:division-hecke-determinant} identifies the
degree-zero determinant with the reciprocal of the relative spherical factor,
\[
L_D(\Gamma,Q^{(n-1)/2}u)^{-1}.
\]
The alternating determinant argument completing the proof of
Theorem~\ref{main} is formal and depends only on the finite cochain
complex of pointed facets.  Applying it to \(\scrB_{D,\Gamma}\) proves the
displayed identity.
\end{proof}

\section*{Acknowledgments}
The authors acknowledge the use of ChatGPT for editorial assistance in improving
the exposition and organization of this manuscript.

\end{document}